\documentclass[11pt,english]{amsart}
\usepackage[T1]{fontenc}
\usepackage[latin9]{inputenc}
\usepackage{amstext}
\usepackage{amsthm}
\allowdisplaybreaks

\makeatletter
\usepackage{amsmath,amsfonts,amssymb,amsthm,epsfig}
\usepackage{graphicx}
\usepackage{soul}

\usepackage{color}
\usepackage[final,colorlinks,linkcolor=blue,anchorcolor=red,citecolor=blue]{hyperref}
\def\R{\mathbb R}
\def\N{\mathbb N}

\def\S{\mathbb S}
\newcommand{\sing}{\text{\rm sing}}

\def\de{\delta}

\def\la{\lambda}

\def\om{\omega}
\def\na{\nabla}
\def\Om{\Omega}  % domains
\def\Si{\Sigma}      % Sum
\def\La{\Lambda} 
\def\e{\varepsilon}     
\newcommand{\supp}{\operatorname{supp}}
\def\cal{\mathcal}
\def\wq{\infty}
\def\pa{\partial}

\def\loc{\text{\rm loc}}

\def\Div{\mathrm{Div}}

\newcommand{\D}{{\rm d}}%  integral sign d

\newcommand{\medint}{-\kern -,375cm\int}         %  average integral
\newcommand{\medintinrigo}{-\kern -,315cm\int}
\newcommand{\wto}{\rightharpoonup}                %  weak convergence

\newcommand{\LLcorner}{\scalebox{1.5}{\ensuremath{\llcorner}}}   %need package{graphicx}

\numberwithin{equation}{section}
\newtheorem{theorem}{Theorem}[section]
\newtheorem*{theorem*}{Theorem}  %Theorems without number

\newtheorem*{conclusion*}{Conclusin}

\newtheorem*{corollary*}{Corollary}

\newtheorem{definition}[theorem]{Definition}

\newtheorem{lemma}[theorem]{Lemma}
\newtheorem*{lemma*}{Lemma}

\newtheorem*{notation*}{Notation}

\newtheorem{proposition}[theorem]{Proposition}
\newtheorem*{proposition*}{Proposition}

\newtheorem*{remark*}{Remark}

\newtheorem*{example*}{Example}                %%%%% example can not use, due to conflict with Lyx.
\theoremstyle{definition}

\def\Xint#1{\mathchoice
	{\XXint\displaystyle\textstyle{#1}}%
	{\XXint\textstyle\scriptstyle{#1}}%
	{\XXint\scriptstyle\scriptscriptstyle{#1}}%
	{\XXint\scriptscriptstyle\scriptscriptstyle{#1}}%
	\!\int}
\def\XXint#1#2#3{{\setbox0=\hbox{$#1{#2#3}{\int}$ }
		\vcenter{\hbox{$#2#3$ }}\kern-.6\wd0}}

\def\dashint{\Xint-}

\makeatother

\usepackage{babel}
\begin{document}
	\title[]{Energy identity for Intrinsic Stationary Biharmonic Mappings into Homogeneous Spaces in Supercritical Dimensions}
	
	\author[C.-Y. Guo, Y.-T. Kang,  M.-L. Liu and W.-J. Qi]{Chang-Yu Guo, Yu-Ting Kang, Ming-Lun Liu and Wen-Juan Qi}
	
	\address[Chang-Yu Guo]{Research Center for Mathematics and Interdisciplinary Sciences, Shandong University 266237,  Qingdao, P. R. China and Department of Physics and Mathematics, University of Eastern Finland, 80101, Joensuu, Finland} \email{changyu.guo@sdu.edu.cn}
	
	\address[Yu-Ting Kang]{Research Center for Mathematics and Interdisciplinary Sciences, Shandong University 266237,  Qingdao, P. R. China and and  Frontiers Science Center for Nonlinear Expectations, Ministry of Education, P. R. China}
	\email{yuting.kang@mail.sdu.edu.cn}

	\address[Ming-Lun Liu]{Research Center for Mathematics and Interdisciplinary Sciences, Shandong University 266237,  Qingdao, P. R. China and and  Frontiers Science Center for Nonlinear Expectations, Ministry of Education, P. R. China}
	\email{minglunliu2021@163.com}
	
	\address[Wen-Juan Qi]{Department of Mathematical Sciences, Tsinghua University, Beijing, 100084, P. R. China}
	\email{wenjuan.qi@mail.sdu.edu.cn}
	
	\thanks{$^*$Corresponding author: Ming-Lun Liu}
	%\thanks{*: Corresponding author}
	
	\thanks{ All authors are supported by the Young Scientist Program of the Ministry of Science and Technology of China (No.~2021YFA1002200), the National Science Foundation of China (No.~12671095), the Taishan Scholar Project  and the Jiangsu Provincial Scientific Research Center of Applied Mathematics under Grant No. 7707014040A}
	%\thanks{Acknowledgement: The authors would like to thank Prof. C.-Y. Guo for bringing this problem to our attention and helpful suggestions for this article.}

	\begin{abstract}
		In this paper, we consider energy identity for intrinsic stationary biharmonic maps into homogeneous spaces in supercritical dimensions, extending the corresponding result of Hornung-Moser [Anal. PDE. 2012] in critical dimension. The proof follows a similar strategy as that of Lin-Rivi\`ere  [Duke Math. J. 2002]. A key ingredient is a conservation law for intrinsic biharmonic maps into homogeneous spaces, which allow us to derive higher regularity of the map. 
	\end{abstract}
	
	\maketitle
	
	%	{\small
		%		\keywords{\noindent {\bf Keywords:} Energy identity, Conservation law, Intrinsic biharmonic mappings, Defect measure, Supercritical dimension}
		%		\smallskip
		%		\newline
		%		%\subjclass{\noindent {\bf 2020 Mathematics Subject Classification:}  35J48, 35G50, 35B65 }
		%		%\tableofcontents
		%	}
	%	\bigskip
	
	\section{Introduction}
	
	Energy identity for harmonic type maps is a fundamental problem in the study of compactness theory and singular sets for geometric variational problems. In supercritical dimensions, the problem is complicated since the energy concentration set may have positive Hausdorff dimension. The first breakthrough in this direction was achieved by Lin and Rivi\`ere  \cite{Lin-Riviere-2002} for sphere-valued stationary harmonic maps. It is therefore natural to ask whether an analogous energy identity continues to hold for higher order geometric variational problems.
	
	Biharmonic maps, as a natural fourth order extension of harmonic maps, provide one of the most important model problems in this direction. In a recent work, Guo-Wang-Xiang \cite{Guo-Wang-Xiang-2026-EI for ESbiHM} established the energy identity for sphere-valued extrinsic stationary biharmonic maps in supercritical dimensions. The biharmonic energy considered by them is the extrinsic functional
	$$
	E^{ext}(u,\Omega)
	=\int_\Omega |\Delta u|^2\,dx,
	$$
	whose definition depends on the isometric embedding of the target manifold into an Euclidean space. From a geometric point of view, we could also consider an energy that depends only on the intrinsic Riemannian geometry of the target manifold. The purpose of this work is to establish the corresponding energy identity for intrinsic stationary biharmonic maps into compact homogeneous Riemannian manifolds in supercritical dimensions.
	
	Let $\Omega\subset\mathbb R^n$ be a bounded smooth domain and let $(N,h)$ be a smooth compact Riemannian manifold without boundary. For convenience, we fix an isometric embedding $N\hookrightarrow\mathbb R^m$. Following  \cite{Moser-2008-CPDE}, we consider the intrinsic second order energy
	$$
	E(u,\Omega)
	:=\int_\Omega |\nabla Du|^2\,dx,
	$$
	where $\nabla=\nabla^u$ denotes the connection on the pull-back bundle $u^*TN$ induced by the Levi-Civita connection of $N$. The natural energy space associated with this functional is
	$$H^{2}(\Om, N)=\Big\{u\in H^{1}(\Om, \R^m):\nabla Du\in L^2(\Omega, (\R^n)^*\otimes (\R^n)^*\otimes \R^m ), u(x)\in N\text{ for a.e. }x\in \Om \Big\}.$$
	This functional $E$ is intrinsic, in the sense that it is independent of the particular isometric embedding of $N$. On the other hand, its analytic structure is more delicate. Moser \cite{Moser-2008-CPDE} introduced the above geometrically motivated Sobolev space and established the existence of minimizers, while Scheven \cite{Scheven-2009-Poincare} proved, among other results, that stationary maps in $H^2(\Omega,N)$ belong locally to $W^{2,2}\cap W^{1,4}$. Consequently, we may work in the usual $W^{2,2}$ framework.

	For a map $u\in W^{2,2}(\Omega,N)$, we write throughout this article 
	$$
	H_{\alpha\beta}(u)
	:=
	(\nabla Du)_{\alpha\beta}
	=
	\partial_{\alpha\beta}u
	+A(u)(\partial_\alpha u,\partial_\beta u),
	$$
	where $A(\cdot)(\cdot,\cdot)$ denotes the second fundamental form of $N\hookrightarrow\mathbb R^m$. A map $u$ is called a weakly intrinsic biharmonic map if it is a critical point of $E$ with respect to outer variations, i.e.
	$$
	\left.
	\frac{d}{dt}
	\right|_{t=0}
	E\bigl(\Pi_N(u+t\varphi),\Omega\bigr)
	=0
	\qquad
	\text{for all }
	\varphi\in C_0^\infty(\Omega,\mathbb R^m),
	$$
	where $\Pi_N$ denotes the nearest-point projection onto $N$. The Euler-Lagrange equation for $E(u,\Omega)$ is
	$$
	\nabla^N_{\partial_\alpha u}\Delta \partial_\alpha u+ R^N(u)(\nabla^N_{\partial_\alpha u} \partial_\beta u, \partial_\alpha u)\partial_\beta u=0,
	$$
	where $R^N$ is the Riemannian curvature on $N$; see \cite{Hornung-Moser-2012-APDE,Moser-2008-CPDE,Scheven-2009-Poincare}.
	Here and afterward repeated indices are summed from $1$ to $n$.

	Next we recall the stationary condition introduced in \cite[Definition 3.1]{Moser-2008-CPDE}. A weakly intrinsic biharmonic map $u\in W^{2,2}(\Omega,N)$ is called stationary if it is in addition critical with respect to inner variations. 
	
	\begin{definition}[Intrinsic stationary biharmonic maps]\label{def: stationary condition}
		%Let $N$ be a Riemannian manifold embedded in $\mathbb{R}^m$.
		An intrinsic biharmonic map $u\in H^2(\Omega,N)$ is called stationary for $E(u,\Omega)$ if  for all
		$\psi=(\psi^1,...,\psi^m)\in C^\infty_0(\Omega,\mathbb{R}^m)$, it holds
		\begin{equation}\label{eq: stationary condition}
			\int_{\Omega}
			\left(
			(\nabla_{\alpha}Du\cdot \nabla_{\beta}Du) \, \partial_\beta\psi^\alpha
			+
			\frac{1}{2}
			(\nabla_{\alpha}\partial_\beta u\cdot \partial_\gamma u)\,
			\partial_{\alpha\beta} \psi^\gamma
			-
			\frac{1}{4}|\nabla Du|^2 \Div\psi
			\right)dx
			=0.
		\end{equation}
	\end{definition}
	
	A crucial consequence of stationarity is the monotonicity formula established
	by Moser \cite[Theorem 3.1]{Moser-2008-CPDE}. If $u\in H^2(\Omega,N)$ is an intrinsic stationary biharmonic map and $B_R(a) \subset \Omega$, then for $X(x):=x-a$, the expression
	$$
	\begin{aligned}
		\Phi_u(a,r)
		:=&
		\frac{1}{4}r^{4-n}
		\int_{B_r(a)} |\nabla Du|^2dx
		+
		\frac{3}{4}r^{3-n}
		\int_{\partial B_r(a)} |Du|^2d\mathcal{H}^{n-1}\\
		+&
		\frac{1}{4}r^{1-n}
		\int_{\partial B_r(a)}
		\Bigl[
		(n-2)|\partial_X u|^2
		-
		2(\nabla_X\partial_X u)\cdot \partial_X u
		\Bigr]
		d\mathcal{H}^{n-1}
	\end{aligned}
	$$
	is well-defined and monotonously nondecreasing for almost every $r\in(0,R]$. More precisely, for almost all $0<r_1\leq r_2\leq R$, there holds
	\begin{equation}\label{eq: monotonic formula}
		\Phi_u(a,r_2)-\Phi_u(a,r_1)
		=
		\int_{B_{r_2}(a)\setminus B_{r_1}(a)}
		\left(
		\frac{|\nabla \partial_X u|^2}{|x-a|^{n-2}}
		+
		(n-2)\frac{|\partial_X u|^2}{|x-a|^n}
		\right)
		dx.
	\end{equation}

	The regularity theory of biharmonic maps is important in the study of energy identity. Chang-Wang-Yang \cite{Chang-W-Y-1999} initiated the study of regularity theory of sphere-valued (i.e., $N=\S^k$) weakly biharmonic maps and established their smoothness
	when $n=4$; see also the work of P. Strzelecki \cite{Strzelecki-2003} with a different method. Soon after that, Wang developed a regularity theory of biharmonic maps into general compact Riemannian manifolds in a series of papers \cite{Wang-2004-CV, Wang-2004-MZ, Wang-2004-CPAM} via the method of Coulomb moving frames. In the critical dimension $n=4$, Lamm and Rivi\`ere \cite{Lamm-Riviere-2008} introduced a class of more general fourth order critical elliptic systems, including biharmonic maps, and then proved the continuity of weak solutions via the conservation law approach initially developed in the celebrated work of Rivi\`ere \cite{Riviere-2007}. Exploring this method, Guo and Xiang \cite{Guo-Xiang-2019-Boundary} proved H\"older continuity of weak solutions, and Guo-Xiang-Zheng \cite{Guo-Xiang-Zheng-2021-CVPDE} established sharp $L^p$-regularity for fourth-order elliptic systems. In supercritical dimensions $n\geq5$, Struwe \cite{Struwe-2008} established partial regularity of biharmonic maps via an alternative fourth order elliptic system, and later Guo, Wang and Xiang \cite{Guo-Wang-Xiang-2023-CVPDE} obtained optimal partial $L^p$-regularity for Struwe's system. Recently, Guo-Jiang-Xiang-Zheng \cite{Guo-J-Xiang-Zheng-2025} proved the optimal higher regularity for extrinsic biharmonic maps via the powerful method of quantitative stratification of Naber-Valtorta \cite{Naber-Val-2017-AnnMath}. This result was soon extended to intrinsic biharmonic maps in \cite{He-Xiang-Zheng-2025}.
	
	For intrinsic stationary biharmonic maps, Moser \cite{Moser-2008-CPDE} proved the following $\varepsilon$-regularity theorem: there exists a constant $\e>0$ such that every stationary intrinsic  biharmonic map $u\in H^{2}(B^n_r(x_0), N)$ with$$
	r^{4-n}\int_{B^n_r(x_0)}\left(|\nabla Du|^2+r^{-2}|Du|^2\right)dx\le \e
	$$
	satisfies $u\in C^\infty (B^n_{r/2}(x_0), N)$.
	As a direct corollary of this result, Scheven \cite{Scheven-2009-Poincare} proved that if $u$ is a stationary biharmonic map, then $\mathcal{H}^{n-4}(\sing(u))=0$, where 
	$$
	\sing(u):=\Big\{x\in \Omega: u \text{ is not continuous at any neighborhood of }x\Big\}.
	$$
	The monotonicity formula and the $\e$-regularity theorem also lead to a deep defect measure theory, extending Lin's seminal work \cite{Lin-1999-Annals} for stationary harmonic maps. More precisely, let $\{u_k\}$ be a sequence of stationary intrinsic biharmonic maps with uniformly bounded energies and suppose that
	$$
	u_k\rightharpoonup u
	\qquad\text{weakly in }W^{2,2}(\Omega,N).
	$$
	After passing to a subsequence, one has
	$$
	|\nabla Du_k|^2\,dx
	\rightharpoonup
	|\nabla Du|^2\,dx+\nu
	$$
	in the sense of weak convergence of Radon measures, where $\nu$ is a nonnegative defect measure. Moreover, there exists a countably $(n-4)$-rectifiable set $\Sigma\subset\Omega$ such that
	\begin{equation}\label{eq:defect-intro}
		\nu
		=
		\Theta_\nu^{n-4}(x)
		\mathcal H^{n-4}\lfloor\Sigma.
	\end{equation}
	%Based on the defect measure theory, He-Xiang-Zheng \cite{He-Xiang-Zheng-2025} obtained the optimal higher regularity for intrinsic biharmonic maps.

	The energy density $\Theta_\nu^{n-4}(x)$ measures the amount of intrinsic biharmonic energy lost in the concentration set. Thus, as in the harmonic map theory, the natural energy quantization problem is to characterize this density in terms of smooth bubble maps in the critical dimension. Following Sacks and Uhlenbeck's celebrated work on bubbling analysis \cite{Sacks-Uhlenbeck-1981}, energy identities for harmonic type maps in critical dimensions were established in dimension two \cite{DingTian1995,Parker1996,QingTian1997,LinWang1998} and later extended by Lin and Rivi\`ere \cite{Lin-Riviere-2002} to sphere-valued stationary harmonic maps in supercritical dimensions $n\geq3$; see also \cite{zhang-zhu-2026} for an extension of this to homogeneous manifolds $N$.

	Energy quantization for intrinsic biharmonic maps was first understood in the critical dimension $4$. In this case, the concentration set is locally finite, Hornung-Moser \cite{Hornung-Moser-2012-APDE} established an energy identity for intrinsic biharmonic maps into arbitrary compact Riemannian manifolds. Related energy identities for extrinsic biharmonic maps and extensions were obtained in \cite{Chen-Zhu-2023-SCM,Guo-Qi-Sun-Wang-2026-AMS,l-r-4,Liu-Yin-2016-MZ,Wang-Zheng-2012,Wang-Zheng-2013}. The supercritical dimensional case is substantially more difficult because the concentration set only has vanishing $\mathcal{H}^{n-4}$ Hausdorff measure. For extrinsic stationary biharmonic maps into spheres, Guo-Wang-Xiang \cite{Guo-Wang-Xiang-2026-EI for ESbiHM} recently overcame this difficulty and obtained the energy identity by adapting the strategy of Lin-Rivi\`ere \cite{Lin-Riviere-2002}. It is worth pointing out that, Naber and Valtorta \cite{Naber-V-24-arXiv} succeeded in establishing the energy identity for stationary harmonic maps into general closed manifolds $N$ via a significantly new approach. However, due to higher order nature, it seems difficult to adapt their approach to biharmonic maps into general closed manifolds. 
	
	In this article, our first main result establish the energy identity for intrinsic stationary biharmonic mappings into homogeneous target manifolds. For simplicity, we only concern the result on unit ball $B_1^n\subset \R^n$.
	\begin{theorem}\label{thm: main results}
		Let $N$ be a smooth, compact, homogeneous Riemannian manifold. Suppose $u_{k}\in H^{2}(B_1^{n},N)$ is a sequence of intrinsic stationary biharmonic mappings satisfying
		$$
		u_{k}\wto u\quad\text{weakly but not strongly in }H^{2}(B_1^{n})
		$$
		and
		$$
		|\nabla Du_{k}|^{2}dx\wto|\nabla Du|^{2}dx+\nu\quad\text{as weak convergence of Radon measures on $B_1^n$}
		$$
		with
		$$
		\nu=\Theta_{\nu}^{n-4}(x){\cal H}^{n-4}\LLcorner\Sigma
		$$
		for some $(n-4)$-rectifiable set $\Sigma\subset B_1^{n}$. Then for $\mathcal{H}^{n-4}$ a.e. $x\in\Sigma$,
		there exists an integer $l_x$ such that
		$$
		\Theta_{\nu}^{n-4}(x)=\sum_{i=1}^{l_{x}}E(\omega_{i},\R^{4}),
		$$
		where each $\omega_{i}:\R^{4}\to N$ is a nontrivial intrinsic smooth biharmonic map.
	\end{theorem}

	In Theorem \ref{thm: main results}, we may equivalently write
	$$
	\Theta_{\nu}^{n-4}(x)=\sum_{i=1}^{l_{x}}E(\bar{\omega}_{i},\S^{4})\qquad\text{for }{\cal H}^{n-4}\text{-a.e. }x\in\Si,
	$$
	where each $\bar{\omega}_i:\S^{4}\to\S^{m-1}$ is a nontrivial intrinsic biharmonic map defined on $\S^4$.
	
	The general strategy for the proof of Theorem \ref{thm: main results} is similar to that used by Guo-Wang-Xiang \cite{Guo-Wang-Xiang-2026-EI for ESbiHM} for extrinsic biharmonic maps, whereas the idea dates back to Lin-Rivi\'ere \cite{Lin-Riviere-2002} for harmonic maps. It consists of the following four main steps:
	\begin{enumerate}
		\item Construct the first nontrivial bubble by applying defect measure theory.
		
		\item $W^{3,(4/3,1)}$-estimate for $u_k$.
		
		\item A reduction lemma so that we may reduce the energy identity from supercritical dimension to the critical dimension $4$.
		
		\item $L^{(2,\infty)}$ estimate on the neck domain so that we may apply the Lorentz duality to conclude vanishing energy on neck domains.
	\end{enumerate}
	Comparing with the extrinsic biharmonic case \cite{Guo-Wang-Xiang-2026-EI for ESbiHM}, the main difference lies in Steps 1 and 2. In step 1, we shall follow the approach used by Lin \cite{Lin-1999-Annals} for stationary harmonic maps, where regularity theory for intrinsic stationary biharmonic maps developed in \cite{Moser-2008-CPDE,Scheven-2008-ACV,Scheven-2009-Poincare} turns out to be useful. To achieve $W^{3,(4/3,1)}$-estimate, we follow the idea of H\'elein \cite{Helein-1991} to rewrite the biharmonic map equation in divergence form, using Killing vector fields. 
	
	%Our second main result establishes a useful conservation law for intrinsic biharmonic maps when the target manifold is homogeneous. This structure will play an important role in the proof of the energy identity. 
	%Let $N$ be a smooth compact homogeneous Riemannian manifold. Then there exists a Lie group acting transitively on $N$ by isometries, and hence $N$ admits sufficiently many Killing vector fields

	%Thanks to this decomposition and the Killing equations, we obtain the following conservation form of the intrinsic biharmonic mapping equation.

	\begin{theorem}\label{thm: homo conservation law}
		Let $N$ be a smooth, compact, homogeneous Riemannian manifold embedded in $\mathbb{R}^m$. Let $X_1,\ldots,X_M$ and $Y_1,\ldots,Y_M$ be as in Lemma \ref{lemma: homo space Killing}, and let $u \in H^2(\Omega,N)$ be a critical point of $E$ in the sense that
		$$
		\left.\frac{d}{dt}\right|_{t=0}
		E\bigl(\Pi_N(u+t\phi)\bigr)=0
		\qquad
		\text{for all } \phi \in C^\infty_0(\Omega,\mathbb{R}^m).
		$$
		Then $u$ satisfies the conservation law
		\begin{equation}\label{eq: homo conservation}
			\begin{aligned}
				\sum_{\alpha,\beta=1}^n
				\partial_\beta\partial_\alpha(H_{\alpha\beta}(u)\cdot Z(u)) = 2\sum_{\alpha,\beta=1}^n
				\partial_\beta(H_{\alpha\beta}(u)\cdot \nabla^N_{\partial_\alpha u} Z(u))
			\end{aligned}
		\end{equation}
		for all smooth Killing fields $Z$ on $N$.
	\end{theorem}

	%The use of Killing fields for intrinsically biharmonic maps goes back to Hornung-Moser \cite{Hornung-Moser-2012-ACV}. 
	Theorem \ref{thm: homo conservation law} may be viewed as a biharmonic analogue of the conservation law of harmonic maps into homogeneous target by H\'elein \cite{Helein-1991}. For sphere-valued harmonic maps, this was first discovered independently by Chen \cite{Chen-89} and Shatah \cite{Shatah-88}. Later, Wang \cite{Wang-2004-CV} extended this conservation law for biharmonic maps into spheres (note however that his definition of intrinsic biharmonic maps are different from ours). 
	
	The equation \eqref{eq: homo conservation}, together with the fact that the Killing fields span the tangent spaces of $N$, gives a system whose divergence and curl structures are suitable for Hodge decomposition and Lorentz space estimates. We shall explore this fact in Section 3.2 below to achieve higher regularity of stationary intrinsic biharmonic maps.

	To conclude the introduction, let us point out that there is another intrinsic energy functional defined by $$
	\tilde{E}(u,\Omega)=\int_\Omega|\tau(u)|^2dx=\int_\Omega|\mathrm{trace}(\nabla Du)|^2dx.
	$$ 
	However, as pointed out by Moser in \cite[Last paragraph in Section 4]{Moser-2008-CPDE}, one cannot generally expect critical points of $\tilde{E}(u,\Omega)$ to satisfy a monotonicity formula unless the target manifold $N$ has nonpositive sectional curvature. Thus it remains as an open problem whether one can prove Theorem \ref{thm: main results} for stationary points of $\tilde{E}$. 

	\section{Conservation law}
	
	%In this section, we prove Theorem \ref{thm: homo conservation law}. 
	
	\subsection{Homogeneous spaces}

	First of all, we review some basic definitions and properties for homogeneous spaces.
	
	\begin{definition}
		A \emph{left action} of a group $G$ on a manifold $\mathcal{N}$ is a smooth map
		\[
		\lambda : G \times \mathcal{N} \to \mathcal{N}
		\]
		such that
		\[
		\lambda(e,x)=x \quad \text{and} \quad \lambda(ab,x)=\lambda(a,\lambda(b,x))
		\]
		for all $a,b\in G$ and $x\in \mathcal{N}$.
		
		An action is called \emph{transitive} if for any $x,y\in \mathcal{N}$, there exists $a\in G$ such that $\lambda(a,x)=y$.
	\end{definition}
	
	For simplicity, we denote $\lambda(a,x)$ by $a\cdot x$ or $ax$ if there is no confusion. For $x\in \mathcal{N}$, the \emph{isotropy group} at $x$ is defined by
	\[
	G_x := \{a\in G : a\cdot x = x\}.
	\]
	
	%\begin{proposition}
	Let $G\times \mathcal{N}\to \mathcal{N}$ be a transitive action of a Lie group $G$ on a manifold $\mathcal{N}$, and let $K=G_x$ be the isotropy subgroup of a point $x\in \mathcal{N}$. Then
	\begin{enumerate}%[(a)]
		\item the subgroup $K$ is a closed subgroup of $G$;
		\item the natural map $j:G/K\to \mathcal{N}$ given by
		\[
		j(aK)=a\cdot x
		\]
		is a diffeomorphism;
		\item the dimension of $G/K$ is $\dim G-\dim K$.
	\end{enumerate}
	%\end{proposition}
	
	\begin{definition}
		A \emph{homogeneous space} is a manifold $\mathcal{M}$ with a transitive action of a Lie group $G$. Equivalently, it is a manifold of the form $G/K$, where $G$ is a Lie group and $K$ is a closed subgroup of $G$.
	\end{definition}
	
	If $(N,h)$ is a Riemannian manifold, we may consider the isometry group of $N$, denoted by $I(N)$, which is a Lie group.
	
	\begin{definition}
		A \emph{Riemannian homogeneous space} is a Riemannian manifold $(N,h)$ on which its isometry group $I(N)$ acts transitively.
	\end{definition}
	
	%\begin{proposition}
	Let $N$ be a Riemannian homogeneous space. Then the isotropy subgroup of a given point is a compact subgroup of $I(N)$.
	%\end{proposition}
	Hence, a Riemannian homogeneous space $N$ is diffeomorphic to a homogeneous space $G/K$, where $G=I(N)$ and $K$ is the isotropy subgroup of a point.
	
	From now on, $(N,h)$ will be a closed Riemannian manifold on which a Lie group $G$ of dimension $M$ acts transitively by isometries. Let $X$ be a Killing tangent vector field on $N$. Then we have the following useful lemma.
	\begin{lemma}[{\cite[Lemma 2]{Helein-1991}}]\label{lemma: homo space Killing}
		There exist smooth Killing fields $X_1,...,X_M$ on $N$ and smooth tangent vector fields $Y_1,...,Y_M$ on $N$ such that, for all $p\in N$ and all tangent vectors $V\in T_pN$ we have
		\[
		V=\sum_{k=1}^{M}\bigl(V\cdot X_k(p)\bigr)Y_k(p).
		\]
	\end{lemma}

	Denote by $X_1, \ldots, X_M$ and $Y_1, \ldots, Y_M$ the vector fields on $N$ obtained in the conclusion of Lemma \ref{lemma: homo space Killing}. Define the functions $G_{kl}: N \to \mathbb{R}$ by setting
	\[
	G_{kl}(p) = Y_k(p)\cdot Y_l(p).
	\]
	Then we have the following two elementary identities. 
	\begin{lemma}[{\cite[Lemma 2.1]{Hornung-Moser-2012-ACV}}]\label{lemma:21}
		Let $p \in N$, let $V \in T_p N$ and $Z$ be a smooth Killing field on $N$. Then
		\[
		\frac{1}{2}\sum_{k,l = 1}^{M}(D_Z G_{kl})(p)(V\cdot X_k(p))(V\cdot X_l(p))
		= -\sum_{k = 1}^{M}(V\cdot Y_k(p)) V\cdot [Z,X_k](p).
		\]
	\end{lemma}
	
%	\begin{lemma}[{\cite[Lemma 2.2]{Hornung-Moser-2012-ACV}}]\label{lemma:22}
%		Assume $u\in H^2(\Omega,N)$, let $Z$ be a smooth Killing field on $N$. Then
%		\[
%		\Div(\nabla u\cdot Z(u))=\tau(u)\cdot Z(u)=\Delta u\cdot Z(u).
%		\]
%	\end{lemma}

	In our later proof, we shall use the following simple fact from  linear algebra.
	\begin{lemma}\label{lemma: Killing another formula}
		If there exist vectors $X_1,...,X_M,Y_1,...,Y_M\in \R^m$ such that for any vector $V\in \R^m$
		%\begin{equation}\label{eq: Killing XY}
		$$V=\sum_{k=1}^{M}(V\cdot X_k) Y_k,$$
		%\end{equation}
		then
		%\begin{equation}\label{eq: Killing YX}
		$$	V=\sum_{k=1}^{M}(V\cdot Y_k) X_k.$$
		%\end{equation} 
	\end{lemma}
	
	\begin{proof}
		Define matrix $X=(X_1,...,X_M)_{m\times M}$ and $Y=(Y_1,...,Y_M)_{m\times M}$. Then for any $V\in\R^m$, $$
		V=YX^TV.
		$$
		This implies $YX^T=E_{m\times m}$, the identity matrix. Thus, $$
		V= E_{m\times m}V= (YX^T)^TV= XY^TV.
		$$
	\end{proof}
	
	\subsection{Proof of Theorem \ref{thm: homo conservation law}}
	\begin{proof}[Proof of Theorem \ref{thm: homo conservation law}]
		
		For every $p\in N$ and $V\in T_pN$, define the function $G_{kl}:N\to\R$ as $$
		G_{kl}(p)=Y_k(p)\cdot Y_l(p).
		$$
		Then
		%\begin{equation}\label{eq: V inner product}
		\[
		\begin{aligned}
			|V|^2=&\left(\sum_{k=1}^M (V\cdot X_k(p))Y_k(p)\right)\cdot \left(\sum_{l=1}^M (V\cdot X_l(p))Y_l(p)\right)\\
			=&\sum_{k,l=1}^{M} G_{kl}(p)(V\cdot X_k(p))(V\cdot X_l(p)).
		\end{aligned}
		\]
		%\end{equation}
		Hence the energy functional becomes 
		%\begin{equation}\label{eq: energy rewrite}
		\[
		\int_{\Omega}
		\sum_{\alpha,\beta=1}^n\sum_{k,l=1}^M
		G_{kl}(u)
		(H_{\alpha\beta}(u)\cdot X_k(u))
		(H_{\alpha\beta}(u)\cdot X_l(u))\,dx,
		\]
		%\end{equation}
		where $H_{\alpha\beta}(u)=\partial_{\alpha\beta}u+A(u)(\partial_\alpha u,\partial_\beta u)\in T_uN$. Moreover, Lemma \ref{lemma:21} gives 
		\begin{equation}\label{eq: HornungMoser 2.1}
			\frac12\sum_{k,l=1}^M D_ZG_{kl}(p)(V\cdot X_k(p))(V\cdot X_l(p))
			=-\sum_{k=1}^M(V\cdot Y_k(p))\left(V\cdot[Z,X_k](p)\right).
		\end{equation}

		For each Killing field $X_k$, the product rule gives both
		\begin{equation*}
			\begin{aligned}\label{eq: Killing-01}
				\partial_\alpha
				(\partial_\beta u\cdot X_k(u))&=(\partial_{\alpha\beta}u)\cdot X_k(u)
				+\partial_\beta u\cdot D_{\partial_\alpha u}X_k\\
				&=H_{\alpha\beta}(u)\cdot X_k(u)
				+\partial_\beta u\cdot\nabla_{\partial_\alpha u}^NX_k.
			\end{aligned}
		\end{equation*}
		and
		\begin{equation*}\label{eq: Killing-02}
			\partial_\beta
			(\partial_\alpha u\cdot X_k(u))
			=H_{\beta\alpha}(u)\cdot X_k(u)
			+\partial_\alpha u\cdot\nabla_{\partial_\beta u}^NX_k.
		\end{equation*}
		Using $H_{\alpha\beta}=H_{\beta\alpha}$ and the Killing equation for $X_k$, we obtain
		\begin{equation}\label{eq: Hab symmetric}
			H_{\alpha\beta}(u)\cdot X_k(u)
			=\frac12\left[
			\partial_\alpha(\partial_\beta u\cdot X_k(u))
			+\partial_\beta(\partial_\alpha u\cdot X_k(u))
			\right].
		\end{equation}

		Fix a smooth Killing field $Z$ on $N$ and
		$\varphi\in C_0^\infty(\Omega)$, define the variation
		%\begin{equation}\label{eq: variation}
		$$u_t=\Pi_N(u+t\varphi Z(u)).$$
		%\end{equation}
		Then
		$$\frac{\partial u_t}{\partial t}\bigg|_{t=0}
		=\varphi Z(u),
		$$
		$$
		\begin{aligned}
			\frac{\partial}{\partial t}\bigg|_{t=0}\left(\partial_\beta u_t\right)=&(\partial_\beta\varphi)Z(u)+\varphi D_{\partial_\beta u}^NZ\\
			=&(\partial_\beta\varphi)Z(u)+\varphi \nabla^N_{\partial_\beta u}Z- \varphi A(u)(\partial_\beta u,Z).
		\end{aligned}
		$$
		Therefore
		%\begin{equation}\label{eq: variation-01}
		\[
		\begin{aligned}
			\frac{\partial}{\partial t}\bigg|_{t=0}\big(\partial_\beta u_t\cdot X_l(u_t)\big)
			&=(\partial_\beta\varphi)(Z\cdot X_l)(u)
			+\varphi X_l(u)\cdot\nabla_{\partial_\beta u}^NZ
			+\varphi\partial_\beta u\cdot\nabla_Z^NX_l\\
			&= (\partial_\beta\varphi)(Z\cdot X_l)(u)
			+\varphi\partial_\beta u\cdot[Z,X_l](u),
		\end{aligned}
		\]
		%\end{equation}
		where in the last equality we used the Killing equation for $Z$
		$$
		X_l\cdot\nabla_{\partial_\beta u}^NZ
		=-\partial_\beta u\cdot\nabla_{X_l}^NZ.
		$$
		Applying \eqref{eq: Hab symmetric} to $u_t$ and differentiating at $t=0$ gives
		\begin{equation}\label{eq: variation-02}
			\begin{aligned}
				\frac{\partial}{\partial t}\bigg|_{t=0}
				[H_{\alpha\beta}(u_t)\cdot X_l(u_t)]
				=&\frac12\partial_\alpha\left[(\partial_\beta\varphi)(Z\cdot X_l)(u)+\varphi\partial_\beta u\cdot[Z,X_l](u)\right]\\
				&+\frac12\partial_\beta\left[(\partial_\alpha\varphi)(Z\cdot X_l)(u)+\varphi\partial_\alpha u\cdot[Z,X_l](u)\right].
			\end{aligned}
		\end{equation}

		Now, differentiating \eqref{eq: HornungMoser 2.1} with respect to $t$ at $t=0$, we obtain three terms:
		\[
		\frac{1}{2}\left.\frac{d}{dt}\right|_{t=0}E(u_{t}) = I_{1} + I_{2} + I_{3},
		\]
		where
		\[
		I_{1} = \frac{1}{2}\int_{\Omega}\varphi \sum_{\alpha,\beta=1}^n\sum_{k,l=1}^M D_{Z}G_{kl}(u)(H_{\alpha\beta}(u)\cdot X_{k}(u))(H_{\alpha\beta}(u)\cdot X_{l}(u))dx
		\]
		comes from the variation of $G_{kl}(u_{t})$, and
		\[
		I_{2} = \frac{1}{2}\int_{\Omega}\sum_{\alpha,\beta=1}^n\sum_{k,l=1}^M G_{kl}(u)(H_{\alpha\beta}\cdot X_{k})\frac{\partial}{\partial t}\bigg|_{t=0}[H_{\alpha\beta}(u_t)\cdot X_{l}(u_t)]dx,
		\]
		\[
		I_{3} = \frac{1}{2}\int_{\Omega}\sum_{\alpha,\beta=1}^n\sum_{k,l=1}^M G_{kl}(u)(H_{\alpha\beta}\cdot X_{l})\frac{\partial}{\partial t}\bigg|_{t=0}[H_{\alpha\beta}(u_t)\cdot X_{k}(u_t)]dx
		\]
		come from the variations of the two factors $H_{\alpha\beta}\cdot X_{k}$ and $H_{\alpha\beta}\cdot X_{l}$. By symmetry, $I_{2}=I_{3}$. Thus
		\[
		\frac{1}{2}\left.\frac{d}{dt}\right|_{t=0}E(u_{t}) = I_{1} + 2I_{2}.
		\]
		
		Using \eqref{eq: HornungMoser 2.1}, we compute
		\begin{equation}\label{eq:I1}
			I_{1} = -\int_{\Omega}\varphi \sum_{\alpha,\beta=1}^n\sum_{l=1}^M(H_{\alpha\beta}(u)\cdot Y_{k}(u))(H_{\alpha\beta}(u)\cdot [Z,X_{l}](u))dx. \tag{A}
		\end{equation}
		
		Note that 
		$$
		H_{\alpha\beta}(u)\cdot Y_l(u)=\sum_{k=1}^{M}G_{kl}(u)
		(H_{\alpha\beta}(u)\cdot X_k(u)).
		$$
		Using this and \eqref{eq: variation-02}, we obtain
		\[
		\begin{aligned}
			2I_{2}
			=& \int_{\Omega}\sum_{\alpha,\beta=1}^n\sum_{l=1}^M
			(H_{\alpha\beta}(u)\cdot Y_l(u))
			\partial_\alpha\left[(\partial_\beta\varphi)(Z\cdot X_l)(u)\right]dx\\
			&+ \int_{\Omega}\sum_{\alpha,\beta=1}^n\sum_{l=1}^M
			(H_{\alpha\beta}(u)\cdot Y_l(u))
			(\partial_\beta u\cdot[Z,X_l](u))(\partial_\alpha\varphi)dx\\
			&+ \int_{\Omega}\varphi\sum_{\alpha,\beta=1}^n\sum_{l=1}^M
			(H_{\alpha\beta}(u)\cdot Y_l(u))
			\partial_\alpha\left[\partial_\beta u\cdot[Z,X_l](u)\right]dx.
		\end{aligned}
		\]
		
		Since $[Z,X_{l}]$ is a Killing field (as both $Z$ and $X_{l}$ are Killing), we can apply \eqref{eq: Hab symmetric} to obtain
		\[
		H_{\alpha\beta}(u)\cdot [Z,X_{l}](u) = \frac{1}{2}[\partial_{\alpha}(\partial_{\beta}u\cdot [Z,X_{l}](u)) + \partial_{\beta}(\partial_{\alpha}u\cdot [Z,X_{l}](u))].
		\]
		Using the symmetry of $H_{\alpha\beta}\cdot Y_{l}$ in $\alpha,\beta$, we have
		\[
		\sum_{\alpha,\beta}(H_{\alpha\beta}(u)\cdot Y_{l}(u))\partial_{\alpha}[\partial_{\beta}u\cdot [Z,X_{l}](u)]
		= \sum_{\alpha,\beta}(H_{\alpha\beta}(u)\cdot Y_{l}(u))(H_{\alpha\beta}(u)\cdot [Z,X_{l}](u)).
		\]
		Thus we have 
		%\begin{equation}\label{eq: variation-04}
		\[
		\begin{aligned}
			&\int_{\Omega}\varphi\sum_{\alpha,\beta=1}^n\sum_{l=1}^M
			(H_{\alpha\beta}(u)\cdot Y_l(u))
			\partial_\alpha\left[\partial_\beta u\cdot[Z,X_l](u)\right]dx\\
			=& \int_{\Omega}\varphi\sum_{\alpha,\beta=1}^n\sum_{l=1}^M
			(H_{\alpha\beta}(u)\cdot Y_l(u))
			(H_{\alpha\beta}(u)\cdot[Z,X_l](u))dx,
		\end{aligned}
		\]
		%\end{equation}
		which cancels \eqref{eq:I1}.
		
		Note also that by Lemma \ref{lemma: Killing another formula}, we have 
		\[
		\sum_{l=1}^M(H_{\alpha\beta}(u)\cdot Y_{l}(u))(Z\cdot X_{l})(u) = H_{\alpha\beta}(u)\cdot Z(u).
		\]
		Therefore, we conclude 
		%\begin{equation}\label{eq: variation-05}
		\[
		\begin{aligned}
			\frac12\frac{d}{dt}\bigg|_{t=0}E(u_t)
			=&\int_{\Omega}\sum_{\alpha,\beta=1}^n\sum_{l=1}^M
			(H_{\alpha\beta}(u)\cdot Y_l(u))
			\partial_\alpha\left[(\partial_\beta\varphi)(Z\cdot X_l)(u)\right]dx\\
			&+ \int_{\Omega}\sum_{\alpha,\beta=1}^n\sum_{l=1}^M
			(H_{\alpha\beta}(u)\cdot Y_l(u))
			(\partial_\beta u\cdot[Z,X_l](u))(\partial_\alpha\varphi)dx\\
			=& \int_{\Omega}\sum_{\alpha,\beta=1}^n
			(H_{\alpha\beta}(u)\cdot Z(u))(\partial_\alpha\partial_\beta\varphi)dx\\
			&+ \int_{\Omega}\sum_{\alpha,\beta=1}^n\sum_{l=1}^M
			(H_{\alpha\beta}(u)\cdot Y_l(u))
			\bigg(\partial_\alpha(Z\cdot X_l)(u)+\partial_\alpha u\cdot[Z,X_l](u)\bigg)(\partial_\beta\varphi)dx.
		\end{aligned}
		\]
		%\end{equation}
		
		Since $Z(u)$ and $X_l(u)$ are Killing, we have
		$$
		\begin{aligned}
			&\quad\partial_\alpha(Z\cdot X_l)(u)+\partial_\alpha u\cdot[Z,X_l](u)\\
			&=X_l(u)\cdot\nabla^N_{\partial_\alpha u} Z(u)+Z(u)\cdot\nabla^N_{\partial_\alpha u} X_l(u)+\partial_\alpha u\cdot \nabla^N_Z X_l(u)-\partial_\alpha u\cdot \nabla^N_{X_l} Z(u)\\
			&= 2X_l(u)\cdot\nabla^N_{\partial_\alpha u} Z(u),
		\end{aligned}
		$$
		and thus Lemma \ref{lemma: Killing another formula} gives 
		%\begin{equation}\label{eq: variation-06}
		\[
		\begin{aligned}
			&\quad\sum_{l=1}^M
			(H_{\alpha\beta}(u)\cdot Y_l(u))
			\bigg(\partial_\alpha(Z\cdot X_l)(u)+\partial_\alpha u\cdot[Z,X_l](u)\bigg)\\
			&= 2\sum_{l=1}^M
			(H_{\alpha\beta}(u)\cdot Y_l(u))(X_l(u)\cdot\nabla^N_{\partial_\alpha u} Z(u))\\
			&= 2H_{\alpha\beta}(u)\cdot \nabla^N_{\partial_\alpha u} Z(u).
		\end{aligned}
		\]
		%\end{equation}

		Finally, integration by parts yields
		\begin{equation*}%\label{eq: variation-07}
			%\begin{aligned}
			\frac12\frac{d}{dt}\bigg|_{t=0}E(u_t)
			= \int_{\Omega}\varphi\sum_{\alpha,\beta=1}^n
			\partial_\beta\partial_\alpha(H_{\alpha\beta}(u)\cdot Z(u))dx
			- 2\int_{\Omega}\varphi\sum_{\alpha,\beta=1}^n
			\partial_\beta(H_{\alpha\beta}(u)\cdot \nabla^N_{\partial_\alpha u} Z(u))dx,
			%\end{aligned}
		\end{equation*}
		which proves \eqref{eq: homo conservation}.
	\end{proof}

	\section{Construction of the first bubble}
	
	%\subsection{Construction of the first bubble}
	%\textbf{Step 1.} \textbf{Construction of the first bubble (\cite[Section 3]{Scheven-2008-ACV}).}
	In this section, we will outline the blow-up procedure to derive the first bubble. 
	
	Under the same assumptions as in Theorem \ref{thm: main results}, after selecting a generic point $x_0$ in the energy concentration set $\Sigma$
	(such a point exists for ${\cal H}^{n-4}$ -a.e. $x\in \Sigma$)
	and considering the rescaled sequence $u_k(x_0+r_k\cdot)$ at $x_0$ with scalings $r_k\to 0$, we may assume, without loss of generality,
	that $x_0=0$ and  still denote the rescaled sequence of stationary biharmonic
	maps  as $u_{k}\in H^{2}(B_{2}^{n},N)$ with uniformly bounded  energies, that is,
	$\sup_{k\ge1}\|u_{k}\|_{H^{2}(B_{2}^{n})}\le\La$, such that
	\[
	u_{k}\wto c\quad\text{weakly in }H^{2}(B_{2}^{n})\text{ and strongly in }H^{1}(B_{2}^{n}),
	\]
	\[
	u_{k}\to c\quad\text{ in }C_{\loc}^{2}(B_{2}^{n}\backslash P)
	\]
	for some constant $c\in \mathbb{S}^{m-1}$,
	and
	\[
	|\nabla Du_{k}|^{2}dx\wto\Theta^{n-4}_{\nu}(0){\cal H}^{n-4}\LLcorner P
	\]
	as weak convergence of Radon measures, where
	\[
	P=B_{2}^{n-4}\times\{0\}.
	\]
	
	Our main result of this section is the following existence theorem. 
	\begin{theorem}[Construction of first bubble]\label{thm:first bubble}
		If $\nu\neq 0$, then there exists a nonconstant smooth intrinsic biharmonic map $v\colon \R^4\to N$.  	
	\end{theorem}
	
	%By \cite[Lemma 6.6]{Scheven-2009-Poincare}, there exists a sequence of intrinsic stationary biharmonic maps $u_k\in H^2(B_4^n,N)$ with uniformly bounded energies, i.e. $\sup_{k\ge1}\|u_k\|_{H^2(B_4^n)}\le \Lambda$, such that
	%$$
	%u_{k}\wto c\quad\text{weakly in }H^{2}(B_{4}^{n})\text{ and strongly in }H^{1}(B_{4}^{n}),
	%$$
	%$$
	%u_{k}\to c\quad\text{ in }C_{\loc}^{2}(B_{4}^{n}\backslash P),
	%$$
	%$$
	%|\nabla Du_{k}|^{2}dx\wto\Theta^{n-4}_{\nu}(0){\cal H}^{n-4}\LLcorner P\quad\text{ as weak convergence of Radon measures},
	%$$
	%where $c\in N$ is a constant, $P=B_{4}^{n-4}\times\{0\}$. 
	
	\begin{proof}
		
		By \cite[Corollary 3.2]{Moser-2008-CPDE}, we can further assume that
		\begin{equation}\label{eq: uniform morrey norm}
			\sup_{0<r<1,x\in B_{1}^{n}}r^{4-n}\int_{B_{r}^{n}}(|\nabla Du_{k}|^{2}+r^{-2}|Du_{k}|^{2})dx\le C,\quad\forall\,k\ge1,
		\end{equation}
		where $C=C(n,\La)>0$ is a universal constant.
		
		To be clear, we write $x=(X_1, X_2)$, where  $X_{1}=(x_{1},\cdots,x_{n-4})$ and $X_{2}=(x_{n-3},\cdots,x_{n})$. 
		\smallskip 
		
		\noindent\textbf{Claim 1.} For each $j=1,...,n-4$, 
		\begin{equation}\label{eq: Claim 1 eq}
			\int_{B_1^{n}}|\na_j Du_{k}|^{2}\,dx+\int_{B_1^{n}}|Du_{k}|^{2}\,dx\to0.
		\end{equation}
		
		For $j=0,1,...,n$, let $x^0:=0$, $x^j:=e_j$ be the standard basis of $\R^n$. Since $u_k\to c$ strongly in $W^{1,2}$, we have $\Phi_{u_k}(a,r)\to r^{4-n}\nu(B_r(a))$. Thus, by the monotonic formula \eqref{eq: monotonic formula}, for each $j=1,...,n-4$, it holds
		\begin{equation}\label{eq: Claim 1-01}
			\begin{aligned}
				\int_{B_{R}(x^j)\setminus B_{r}(x^j)}
				\frac{|\nabla \partial_X u_k|^2}{|x-x^j|^{n-2}}\le\,& \Phi_{u_k}(x^j,R)-\Phi_{u_k}(x^j,r)\\
				\xrightarrow{k\to\infty}\,& R^{4-n}\nu(B_R(x^j))-r^{4-n}\nu(B_r(x^j))
				=\,0.
			\end{aligned}
		\end{equation}
		Furthermore, by \eqref{eq: uniform morrey norm}, we have 
		$$
		\int_{B_{r}(x^j)}
		|\nabla \partial_X u_k|^2 \le C\int_{B_{r}(x^j)}(r^2|\nabla Du_{k}|^{2}+|Du_{k}|^{2})dx\le Cr^{n-2}.
		$$
		Combined with \eqref{eq: Claim 1-01}, this implies
		\begin{equation}\label{eq: Claim 1-02}
			\int_{B_R(x^j)} |\nabla \partial_X u_k|^2\to0
		\end{equation}
		as $k\to\infty$ for $0<R\le 2$ and $j=1,...,n-4$. By the identity $$
		\partial_ju_k(x)=x^j\cdot Du_k(x)=x\cdot Du_k-(x-x^j)\cdot Du_k=:\partial_{n_0}u_k-\partial_{n_j}u_k,
		$$we have the estimate$$
		\begin{aligned}
			\int_{B_1(0)}|\nabla \partial_ju_k|^2=&\int_{B_1(0)}|\nabla (x^j\cdot Du_k)|^2\\
			%	\le&2\int_{B_1(0)}|\nabla\partial_{n_0}u_k|^2+2\int_{B_1(0)}|\nabla\partial_{n_j}u_k|^2\\
			\le& 2\int_{B_1(x^0)}|\nabla\partial_{X}u_k|^2+2\int_{B_2(x^j)}|\nabla\partial_{X}u_k|^2\to0
		\end{aligned}
		$$
		as $k\to\infty$. This proves Claim 1 since $
		\nabla_{j}Du_k=\nabla \partial_ju_k
		$ and $Du_k\to 0$ strongly in $L^2$.
		~\\
		
		\noindent\textbf{Claim 2.} For all $k\in\N$ sufficiently large, there exist 
		$X_1^k\in B_{1/2}^{n-4}(0)$ and $\delta_k\to 0$ such that
		\begin{equation}\label{eq: Claim 2 eq}
			\max_{X_2\in B_{1/2}^{4}(0)}
			\delta_k^{4-n}\int_{B_{\delta_k}^{n-4}(X_1^k)\times B_{\delta_k}^{4}(X_2)}\left(|\nabla Du_k|^2+\delta_k^{-2}|Du_k|^2\right)dX_1dX_2=\frac{\varepsilon}{c(n)}.
		\end{equation}
		Moreover, the maximal is achieved at some $X_2^k\in B_{1/4}^{4}(0)$.
		
		To prove claim 2, we define a function $f_k:B^{n-4}_{1/2}(0)\to\R$ by $$
		f_k(X_1)=\sum_{j=1}^{n-4}\int_{B^4_{1/2}(0)}|\nabla_j Du_k|^2\,dX_2.
		$$
		Then $\|f_k\|_{L^1(B^{n-4}_{1/2}(0))}\to0$ according to Claim 1. The $\varepsilon$-regularity theory implies that $u_k$ is smooth in a neighborhood of $\{X_1\}\times B^4_{1/2}(0)$ for $\mathcal{H}^{n-4}$ a.e. $X_1\in B^{n-4}_{1/2}(0)$. We may select a sequence $\{X_1^k\}_{k=1}^\infty$ of such points. By the weak $L^1$ estimate for the Hardy-Littlewood maximal function
		$$
		Mf_k(X_1)=\sup_{r>0}r^{4-n}\int_{B_r^{n-4}(X_1)}f_k,
		$$
		we have$$
		\mathcal{H}^{n-4}\{x\in B^{n-4}_{1/2}(0): Mf_k>\|f_k\|_{L^1}^{\frac12}\}\le C\frac{\|f_k\|_{L^1}}{\|f_k\|_{L^1}^{\frac12}}\xrightarrow{k\to \infty}0.
		$$
		Then there exists a set $E\subset B^{n-4}_{1/2}(0)$ with positive $\mathcal{H}^{n-4}$ measure such that
		\begin{equation}\label{eq: Claim 2-00}
			\sup_{r>0} r^{4-n}\int_{B^{n-4}_{r}(X_1^k)}f_k(X_1)dX_1\xrightarrow{k\to \infty} 0.
		\end{equation}
		This implies
		\begin{equation}\label{eq: X-2}
			\sup_{0<r<1/2} r^{4-n}\int_{B^{n-4}_{r}(X_1^k)\times B^4_1}\left(|\nabla \partial_j u_k|^2+r^{-2}|Du_k|^2\right)dX_1dX_2\to 0	
		\end{equation}
		
		For each $k\in\N$, there exists $\delta_k\to0$ such that for all $\delta\in(0,\delta_k)$, 
		\begin{equation}\label{eq: Claim 2-01}
			\delta^{4-n}\int_{B_{\delta}^{n-4}(X_1^k)\times B_{\delta}^{4}(X_2)}\left(|\nabla Du_k|^2+\delta^{-2}|Du_k|^2\right)dX_1dX_2\le \frac{\e}{2c(n)}.
		\end{equation}
		On the other hand, if $\delta>0$ is fixed, then for all $k$ sufficiently large, we have 
		\begin{equation}\label{eq: Claim 2-02}
			\begin{aligned}
				\max_{X_2\in B_{1/2}^{4}(0)}
				\delta^{4-n}\int_{B_{\delta}^{n-4}(X_1^k)\times B_{\delta}^{4}(X_2)}\left(|\nabla Du_k|^2+\delta^{-2}|Du_k|^2\right)dX_1dX_2\ge \varepsilon.
			\end{aligned}
		\end{equation}
		Indeed, if not, then for all $k\in\N$ and $X_2\in B_{1/2}^{4}(0)$, 
		$$
		\delta^{4-n}\int_{B_{\delta}^{n-4}(X_1^k)\times B_{\delta}^{4}(X_2)}\left(|\nabla Du_k|^2+\delta^{-2}|Du_k|^2\right)dX_1dX_2< \varepsilon.
		$$
		By \cite[Corollary 5.4]{Scheven-2009-Poincare}, this would imply that $\delta^2|\nabla Du_k(x)|+\delta|Du_k(x)|\le 1$ for any $x\in B_{\delta/2}^{n-4}(X_1^k)\times B_{\delta}^{4}(0)$ , which contradicts with the fact that $|\nabla Du_k|^2dx\wto \nu$. Combining \eqref{eq: Claim 2-01} and \eqref{eq: Claim 2-02}, we may find a sequence $\{\delta_k\}$ with limit $0$ such that \eqref{eq: Claim 2 eq} holds. 
		
		Next we show that the maximal is achieved at some point $X_2^k\in B_{1/4}^{4}(0)$. If this fails, then the maximal is achieved at $|X_2^k|>1/4$. By \eqref{eq: Claim 2-02} we have $$
		\begin{aligned}
			\int_{B_1^{n-4}(0)\times (B^4_{1/2}\backslash B^4_{1/8}(0))}\left(|\nabla Du_k|^2+|Du_k|^2\right)dX_1dX_2\ge C(n)\e\delta^{n-2}>0.
		\end{aligned}
		$$
		However,  as $|\nabla Du_{k}|^{2}+|Du_k|^2dx\wto\Theta^{n-4}_{\nu}(0){\cal H}^{n-4}\LLcorner P$, we have
		$$
		\begin{aligned}
			\int_{B_1^{n-4}(0)\times (B^4_{1/2}\backslash B^4_{1/8}(0))}\left(|\nabla Du_k|^2+|Du_k|^2\right)dX_1dX_2\to \nu\left(B_1^{n-4}(0)\times (B^4_{1/2}\backslash B^4_{1/8}(0))\right)=0,
		\end{aligned}
		$$
		which is clearly a contradiction. 
		\smallskip 
		
		%\noindent\textbf{Claim 3.} 
		For the scaled map $v_k(y):=u_k((X_1^k,X_2^k)+\delta_ky)$ defined on $B^{n-4}_{R_k}(0)\times B^{4}_{R_k}(0)$ with $R_k=1/(4\delta_{k})$, it follows from \eqref{eq: X-2}, Claim 2 and $u_k\in H^2_\Lambda$ that 
		\begin{enumerate}
			\item  $\sup\limits_{0<r<R_k}\sum\limits_{j=1}^{n-4}r^{4-n}\int_{B^{n-4}_{r}(0)\times B^4_{R_k}(0)}|\nabla_j Dv_k|^2\to 0$ as $k\to\infty$, 
			
			\item $\int_{B^{n-4}_{1}(0)\times B^4_{1}(0)}\left(|\nabla Dv_k|^2+|Dv_k|^2 \right)=\max\limits_{Y_2\in B^4_{R_k-1}(0)}\int_{B^{n-4}_{1}(0)\times B^4_{1}(Y_2)}\left(|\nabla Dv_k|^2+|Dv_k|^2 \right)=\frac{\e}{c(n)}$,
			
			\item $\sup\limits_{k}\int_{B^{n-4}_{R}(0)\times B^4_{R}(0)}\left(|\nabla Dv_k|^2+R^{-2}|Dv_k|^2 \right)\le \Lambda R^{n-4}$.
		\end{enumerate}
		
		%Here (1) follows from \eqref{eq: X-2}, (2) follows from claim 2, (3) follows from $u_k\in H^2_\Lambda$. 
		
		Up to a subsequence, there exists a map $v\in H^2(\R^n,N)$ such that $v_k\to v$ strongly in $W^{1,2}$ and weakly in $W^{2,2}$. Moreover, since $R_i\to\infty$, for any $R>0$, we have 
		$$
		\begin{aligned}
			&\int_{B^{n-4}_{R}(0)\times B^4_{R}(0)}\left(|\nabla Dv|^2+R^{-2}|Dv|^2 \right)\\
			\le&\liminf_{k\to\infty}\int_{B^{n-4}_{R}(0)\times B^4_{R}(0)}\left(|\nabla Dv_k|^2+R^{-2}|Dv_k|^2 \right)\le \Lambda R^{n-4} .
		\end{aligned}
		$$
		
		\noindent\textbf{Claim 3.} $v$ is a non-constant intrinsic biharmonic map on $B_{1/2}^{n-4}(0)\times \R^4$ with
		\begin{equation}\label{eq: Claim 4 eq}
			\int_{B_{R}^{n-4}(0)\times \R^4} \sum_{j=1}^{n-4}|\nabla_j Dv|^2= 0.
		\end{equation}

		Let $\phi \in C_{0}^{\infty}(B_{1}^{n-4}(0)\times B_{1}^{4}(0))$ satisfy $0\leq \phi\leq 1,\phi\equiv 1$ on $B_{3/4}^{n-4}(0)\times B_{1/2}^{4}(0)$, $|\nabla\phi|+|\nabla^2\phi|\leq C$. For $a\in B_{3}^{n-4}(0)\times B_{1}^{4}(0)$, let
		$$
		F_k(a):=
		\int_{B_{1}^{n-4}(0)\times B_{1}^{4}(0)}
		|\nabla Dv_k|^{2}(x+a)\phi(x) dx.
		$$
		Then the stationary condition implies that, for each $j=1,\ldots,n-4$, with $\xi(x)=\phi(x-a)e_j$,
		\begin{align*}
			&\bigl|\partial_{a_j} F_k(a)\bigr|\\
			=&
			\left|
			\int_{B_{1}^{n-4}(0)\times B_{1}^{4}(0)}
			\partial_j\bigl(|\nabla Dv_k|^{2}(x+a)\bigr)\phi(x)
			\right| 
			=\left|
			\int_{B_{1}^{n-4}(a)\times B_{1}^{4}(a)}
			|\nabla Dv_k|^{2}(x)\Div(\xi)
			\right| \\
			\le&	4\left|\int_{B_{1}^{n-4}(0)\times B_{1}^{4}(0)} \left(\nabla_{j}Dv_k\cdot \nabla_{\beta}Dv_k\right)(x+a) \, \partial_\beta\phi(x)\right|\\
			&+ 2\left|\int_{B_{1}^{n-4}(0)\times B_{1}^{4}(0)}\left(H_{\alpha\beta}(v_k) \cdot \partial_j v_k\right)(x+a) \,
			\partial_{\alpha\beta} \phi(x)	\right|\\
			\leq&
			4C\left(\|\nabla_{j}Dv_k\|_{L^2(B_{5}^{n-4}(0))\times B_{3}^{4}(0)}+\|\partial_{j}v_k\|_{L^2(B_{5}^{n-4}(0))\times B_{3}^{4}(0)}\right)\|\nabla Dv_k\|_{L^2(B_{5}^{n-4}(0)\times B_{3}^{4}(0))},
		\end{align*}
		which tends to 0 as $k\to\infty$ uniformly on compact sets. Since $F_k(a)\leq c(n)^{-1}\varepsilon$ for each $a\in B_{1}^{n-4}(0)\times B_{R_k-1}^{4}(0)$ by (2), we infer that, after a covering argument, for all $b\in B_{R_k-3}^{4}(0)$, 
		$$
		\int_{B^{n-4}_3(0)\times B^4_3(0)} \left(|\nabla Dv_k|^2+|Dv_k|^2\right)(Y_1,Y_2+b) \le \e.
		$$
		Then the $\e$-regularity theorem implies that for all $R>0$, up to a subsequence if necessary, $v_k\to v$ in $C^4(B_2^{n-4}(0)\times B_R^4(0))$. The limit map $v$ is smooth intrinsic biharmonic map defined on $B_2^{n-4}(0)\times \R^4$, which is non-constant since  
		$$
		\int_{B^{n-4}_{1}(0)\times B^4_{1}(0)}\left(|\nabla Dv|^2+|Dv|^2 \right)=\frac{\e}{c(n)}.
		$$
		Moreover, by (1) and (3), we have for all $R>0$
		$$
		\int_{B_{R}^{n-4}(0)\times \R^4} \sum_{j=1}^{n-4}|\nabla_j Dv|^2= 0
		$$
		and 
		$$
		\int_{B_R^n}|\nabla Dv|^2\leq \Lambda R^{n-4}. 
		$$

	\end{proof}
	
	\section{Higher regularity via conservation law}

	First, we briefly recall the definition of Lorentz-Sobolev spaces. For a measurable function $f\colon \Om\to\R$, denote its distributional function by $\de_{f}(t)=|\{x\in\Om:|f(x)|>t\}|$ and the nonincreasing rearrangement of $|f|$ by $f^{\ast}(t)=\inf\{s>0:\de_{f}(s)\le t\}$ with $t\ge0$. Define
	\begin{eqnarray*}
		f^{\ast\ast}(t)\equiv\frac{1}{t}\int_{0}^{t}f^{\ast}(s)\D s, \quad t>0.
	\end{eqnarray*}
	The Lorentz space $L^{(p,q)}(\Om)$ ($1<p<\wq,1\le q\le\wq$) is the space of measurable functions $f:\Om\to\R$ such that
	$$
	\|f\|_{L^{(p,q)}(\Om)}\equiv\begin{cases}\displaystyle
		\left(\int_{0}^{\wq}(t^{1/p}f^{\ast\ast}(t))^{q}\frac{\D t}{t}\right)^{1/q}, & \text{if }1\le q<\wq,\\
		\displaystyle\sup_{t>0}t^{1/p}f^{\ast\ast}(t) & \text{if }q=\wq
	\end{cases}
	$$
	is finite.	
	
	The \emph{Lorentz-Sobolev space} $W^{k,(p,q)}(\Om)$ for $1<p<\wq,1\le q\le\wq, k\in\N$ consists of functions $f\in L^{(p,q)}(\Om)$ with $\na^j f\in L^{(p,q)}(\Om)$ for $j=1,..,k$. A natural norm for a Lorentz-Sobolev function $f\in W^{k,(p,q)}(\Om)$ is defined by
	$$
	\|f\|_{W^{k,(p,q)}(\Om)}=\Big(\|f\|_{L^{(p,q)}(\Om)}^{p}+\sum_{j=1}^{k}\|\na^j f\|_{L^{(p,q)}(\Om)}^{p}\Big)^{1/p}.
	$$	
	
	Let $N$ be a smooth, compact, homogeneous Riemannian manifold embedded in $\mathbb{R}^m$ and $u\in H^2(\Omega,N)$ be a weakly intrinsic biharmonic map such that \eqref{eq: homo conservation} is satisfied. Since $u\in H^2(\Omega,N)\cap L^\infty$, for any ball $B\subset\subset \Omega$, $H_{\alpha\beta}(u)\in L^2(B)$, $Du\in L^{(4,2)}(B)$. This implies $|H_{\alpha\beta}(u)||Du|,|Du|^3\in L^{(4/3,1)}(B)$. Moreover, we have the following higher regularity result. 
	
	\begin{proposition}\label{prop:higher regularity}
		%If $u\in W^{2,2}(\Omega,N)$ satisfies \eqref{eq: homo conservation}, then  
		We have $u\in W^{3,(4/3,1)}(B)$ with 
		\begin{equation}\label{eq: u W3,4/3,1}
			\|D^3u\|_{L^{(4/3,1)}(B)}\le C(n,m)\|u\|_{H^2(\Omega)}.
		\end{equation}
	\end{proposition}
	
	%we shall use the following well-known Hodge decomposition for $L^p$ integrable vector fields; see for instance \cite[Chapter 10.5]{Iwaniec-Martin-2001-book}.
	%\begin{lemma}[\cite{Iwaniec-Martin-2001-book}]\label{Hodge decomposition}
	%	Let $p\in(1,\infty)$. Every vector field $V\in L^p(B_r(x_0),\R^n)$ with  $B_r(x_0)\subset\R^n$, can be uniquely decomposed as
	%	$$
	%	V=\nabla a+\nabla^{\perp} b+h,
	%	$$
	%	where $a\in W^{1,p}(B_r(x_0)), b\in W^{1,p}_0(B_r(x_0),\bigwedge^2\R^n)$ with $db=0$, and $h\in C^{\infty}(B_r(x_0),\R^n)$ is harmonic. Moreover, we have the estimate$$
	%	\|a\|_{W^{1,p}(B_r(x_0))}+\|b\|_{W^{1,p}(B_r(x_0))}+\|h\|_{L^p(B_r(x_0))}\leq C\|V\|_{L^p(B_r(x_0))}.
	%	$$	
	%	Here $\nabla^\perp b:=(\delta b)^\sharp\in L^p(B_r(x_0),\R^n)$, where $\delta$ is the formal conjugate operator of $d$ and $\sharp$ is the sharp operator from $\bigwedge^1\R^n$ to $\R^n$.
	%\end{lemma}	

	\begin{proof}
		
		As 
		$$
		|D^3 u|\le C\sum_{\alpha,\beta=1}^n|D(H_{\alpha\beta}(u)-A(u)(\partial_\alpha u,\partial_\beta u))|\\
		\le C\left(|D H||Du|+|D^2u||Du|+|Du|^3\right),
		$$
		it is sufficient to prove the higher regularity $\partial_\gamma(H_{\alpha\beta}(u))\in L^{(4/3,1)}$ for $\gamma=1,...,n$. Then the desired estimate \eqref{eq: u W3,4/3,1} follows directly from this and the conservation law \eqref{eq: homo conservation}. 
		
		By Lemma \ref{lemma: homo space Killing}, we have
		$$
		\begin{aligned}
			\partial_\gamma H_{\alpha\beta}(u)&=\partial_\gamma\left(\sum_{k=1}^{M}\bigl(H_{\alpha\beta}(u)\cdot X_k(u)\bigr)Y_k(u)\right)\\
			&= \sum_{k=1}^{M}\partial_\gamma\bigl(H_{\alpha\beta}(u)\cdot X_k(u)\bigr)Y_k(u)+\sum_{k=1}^{M}\bigl(H_{\alpha\beta}(u)\cdot X_k(u)\bigr)\nabla^N_{\partial_\gamma u}Y_k(u).
		\end{aligned}
		$$
		Note that the second term $(H_{\alpha\beta}(u)\cdot X_k(u))\nabla^N_{\partial_\gamma u}Y_k(u)$ belongs in $L^{(4/3,1)}$.
		%due to the smoothness of $X_k,Y_k$. 
		Thus it remains to prove $H_{\alpha\beta}(u)\cdot X_k(u)\in W^{1,(4/3,1)}$. 
		
		Regarding $(H_{\alpha\beta}(u)\cdot X_k(u))
		_{1\le\alpha\le n}$ as a vector field, standard decomposition for $L^p$ integrable vector fields (see for instance \cite[Chapter 10.5]{Iwaniec-Martin-2001-book}) gives 
		\begin{equation}\label{eq: H Hodge decom}
			H_{\alpha\beta}(u)\cdot X_k(u)=Da+D^\perp b+h,
		\end{equation}
		where $a\in W^{1,2}(B),b=\sum_{1\le i<j\le n}b_{ij}dx^i\wedge dx^j\in W_0^{1,2}(B,\bigwedge^2\R^m)$, $h\in C^\infty(B,\R^m)$ is harmonic. Then we have
		\begin{equation}\label{eq: Delta a}
			\Delta a=\Div (H_{\alpha\beta}(u)\cdot X_k(u))= \sum_{\alpha=1}^{n}\partial_\alpha (H_{\alpha\beta}(u)\cdot X_k(u)),
		\end{equation}
		and 
		\begin{equation}\label{eq: Delta b}
			\Delta b=\Div^{\perp}(H_{\alpha\beta}(u)\cdot X_k(u))=\sum_{1\le i<j\le n}\left( \partial_i (H_{j\beta}(u)\cdot X_k(u))-\partial_j (H_{i\beta}(u)\cdot X_k(u)) \right)dx^i\wedge dx^j,
		\end{equation}
		where the operator $\Div^{\perp}:\R^n\to\bigwedge^2\R^n$ is defined by $\Div^{\perp}V:=d(V^\flat)$,  $\flat$ is the flat operator from $\R^n$ to $\bigwedge^1\R^n$. Note that for any $b\in\bigwedge^2\R^n$ with $db=0$ we have $$
		\Div^{\perp}\circ D^\perp =\Delta_{\R^n} ,\quad \Div^{\perp}\circ D =\Div\circ D^\perp =0.$$
		Direct computation yield
		\begin{equation}\label{eq: RN(u1,u2)u3}
			\begin{aligned}
				&\partial_i\bigl[H_{\beta j}(u)\cdot X_k(u)\bigr]
				-\partial_j\bigl[H_{\beta i}(u)\cdot X_k(u)\bigr]\\
				=&\partial_i\bigl[H_{j\beta}(u)\cdot X_k(u)\bigr]
				-\partial_j\bigl[H_{i\beta}(u)\cdot X_k(u)\bigr]\\
				=&
				R^N(\partial_i u,\partial_j u)\partial_{\beta}u
				\cdot X_k(u)
				+H_{j \beta}\cdot\nabla^N_{\partial_i u}X_k
				-H_{i \beta}\cdot\nabla^N_{\partial_j u}X_k.
			\end{aligned}
		\end{equation}
		H\"older's inequality in Lorentz spaces implies that each term belongs to $L^{(4/3,1)}$. This yields $\Delta b\in L^{(4/3,1)}$ and thus $b\in W^{2,(4/3,1)}$. 
		
		To prove $a\in W^{2,(4/3,1)}$, apply Hodge decomposition to $\big(\sum_{\alpha=1}^{n}\partial_\alpha (H_{\alpha\beta}(u)\cdot X_k(u))\big)_{1\le\beta\le n}$: 
		$$
		\left(\sum_{\alpha=1}^{n}\partial_\alpha (H_{\alpha\beta}(u)\cdot X_k(u))\right)_{1\le\beta\le n}=DA+D^\perp K+\phi,
		$$
		where $A\in W^{1,2}(B),K=\sum_{1\le i<j\le n}K_{ij}dx^i\wedge dx^j\in W_0^{1,2}(B,\bigwedge^2\R^m)$, $\phi\in C^\infty(B,\R^m)$ is harmonic. Then by \eqref{eq: homo conservation}, we have
		\begin{equation}\label{eq: Delta A}
			\Delta A= \Div_\beta\sum_{\alpha}^n
			\partial_\alpha(H_{\alpha\beta}(u)\cdot X_k(u)) = 2\Div_\beta\sum_{\alpha}^n
			(H_{\alpha\beta}(u)\cdot \nabla^N_{\partial_\alpha u} X_k(u))
		\end{equation}
		and
		\begin{equation}\label{eq: Delta K}
			\begin{aligned}
				\Delta K=& \sum_{1\le i<j\le n} \left(\partial_i\left(\sum_{\alpha=1}^{n}\partial_\alpha (H_{\alpha j}(u)\cdot X_k(u))\right)-\partial_j\left(\sum_{\alpha=1}^{n}\partial_\alpha (H_{\alpha i}(u)\cdot X_k(u))\right)\right)dx^i\wedge dx^j\\
				=&\sum_{1\le i<j\le n} \Div_\alpha \left(\partial_i (H_{\alpha j}(u)\cdot X_k(u))-\partial_j (H_{\alpha i}(u)\cdot X_k(u))\right)dx^i\wedge dx^j.
			\end{aligned}
		\end{equation}
		It follows from equation \eqref{eq: RN(u1,u2)u3} and standard elliptic regularity theory in Lorentz spaces that $DA, DK\in L^{(4/3,1)}$. This implies $\Delta a\in L^{(4/3,1)}$ and thus $a\in W^{2,(4/3,1)}$. Finally, we conclude from \eqref{eq: H Hodge decom} that	$D(H_{\alpha\beta}(u)\cdot X_k(u))\in L^{4/3,1}$. 
	\end{proof}
	
	\section{Proof of energy identity theorem}
	
	This section is devoted to the proof of Theorem \ref{thm: main results}. With Theorem \ref{thm:first bubble} and Proposition \ref{prop:higher regularity} at hand, the remaining proof can be proceeded in parallel as that of extrinsic biharmonic maps \cite[Sections 2.2 and 2.3]{Guo-Wang-Xiang-2026-EI for ESbiHM}. For the convenience of readers, we include the proofs here. 
	
	%\subsection{$W^{3,(4/3,1)}$-estimate of $u$}
	Applying \eqref{eq: u W3,4/3,1} to $u_k$ and noticing $\|u_k\|_{H^2(\Omega)}\leq \Lambda$ for all $k$, we obtain
	$$\|u_{k}\|_{W^{3,(\frac{4}{3},1)}(Q_{2/3}^{n})}\le C(n,m)\Lambda,$$
	where $Q_r^n=B_r^{n-4}\times B_r^4$. Hence by Fubini's Theorem
	\begin{equation}\label{eq: Dj uk L(4/j,1)}
		\sum_{j=1}^{3}\left\Vert D^{j}u_{k}(X_{1},\cdot)\right\Vert _{L^{(\frac{4}{j},1)}(B_{2/3}^{4})}\le C(n,m)\Lambda
	\end{equation}
	holds for all $X_{1}\in E_{k}\subset B_{2/3}^{n-4}$, with ${\cal H}^{n-4}(E_{k})\ge0.99{\cal H}^{n-4}(B_{2/3}^{n-4})$, for all sufficiently large $k\ge1$. 
	Furthermore, we may find a subset
	$F_{k}\subset B_{2/3}^{n-4}$, with ${\cal H}^{n-4}(F_{k})\ge0.99{\cal H}^{n-4}(B_{2/3}^{n-4})$,
	such that for all $X_{1}\in F_{k}$,  both \eqref{eq: X-2} and \eqref{eq: X-1} hold.

	\subsection{Reduction to dimension $4$}	
	
	Our main goal is to prove the following reduction lemma.
	\begin{lemma}[Reduction Lemma]\label{lem: main lemma}  For ${\cal H}^{n-4}$-a.e. $X_{1}^{k}\in E_{k}\cap F_{k}$,  there holds
		$$
		\lim_{k\to\wq}\int_{B_{1/2}^{4}(0)}\left|\nabla Du_{k}(X_{1}^{k},X_{2})\right|^{2}\,dX_{2}=\Theta_{\nu}^{n-4}(0).
		$$
	\end{lemma}
	
	%	\st{According to computation below, here should not be $\De u_{i}(X_{1}^{i},X_{2})$,
		%	but should be $\De_{X_{2}}u_{i}(X_{1}^{i},X_{2})=\De\bar{u_{i}}(X_{2})$.}

	\begin{proof}
		
		For $0<\de\ll1$, let $\xi\in C_{0}^{\wq}(\R^{n})$ with ${\rm supp}(\xi)\subset Q_{\de}^n=B_\delta^{n-4}\times B^4_\delta$ and for $a\in Q_{1-\de}^n$, define
		$$
		F_{k}(a):=\int_{Q_{1}^n}|\nabla Du_{k}(x+a)|^{2}\xi(x)\,dx=\int_{Q_{1}^n}|\nabla Du_{k}(x)|^{2}\xi(x-a)\,dx.
		$$
		
		The stationary condition \eqref{eq: stationary condition} of $u_{k}$ implies
		$$
		\int|\nabla Du_{k}|^{2}\,\Div\psi
		=\int\bigl(4\left(\nabla_{\alpha}Du_k\cdot \nabla_{\beta}Du_k\right) \, \partial_\beta\psi^\alpha
		+2\left(H_{\alpha\beta}(u_k)\cdot \partial_\gamma u_k\right)\,
		\partial_{\alpha\beta} \psi^\gamma\bigr)
		$$
		for all $\psi\in C_{0}^{\infty}(Q_{1}^n,\R^{n})$. Hence, by taking $\psi=\xi e_{j}$,
		we deduce
		$$
		\begin{aligned}
			\pa_{a_{j}}F_{k}(a)=&-4\int_{Q_{1}^n}\left(\nabla_{j}Du_k(x+a)\cdot \nabla_{\beta}Du_k(x+a)\right) \, \partial_\beta\xi\\
			&-2\int_{Q_{1}^n}\left(H_{\alpha\beta}(u_k)(x+a)\cdot \partial_j u_k(x+a)\right)\,\partial_{\alpha\beta} \xi
		\end{aligned}
		$$
		and
		$$
		|\pa_{a_{j}}F_{k}(a)|\le C\|\nabla Du_{k}\|_{L^{2}}\left(\|\na_{j} Du_k\|_{L^{2}}+\|\partial_ju_{k}\|_{L^{2}}\right)\|\xi\|_{C^{2}(Q_{1}^n)}.
		$$
		In particular, by \eqref{eq: Claim 1 eq}, for $1\le j\le n-4$ we have
		$$
		\sup_{a\in Q_{1-\de}^n}\left|\pa_{a_{j}}F_{k}(a)\right|\to0\quad\text{ as }k\to\infty.
		$$

		%Let $\xi(X_{1},X_{2})=\xi_{1}^{\eta}(X_{1})\xi_{2}(X_{2})$, where $0\le\xi_{1}\in C_{0}^{\infty}(B_1^{n-4}),0\le\xi_{2}\in C_{0}^{\infty}(B_1^{4})$, and $\xi_{1}^{\eta}=\eta^{4-n}\xi_{1}(\cdot/\eta)$ for some $\eta>0$ sufficiently small. Define function
		%$$
		%F_{k}^{\eta}(a):=\int_{Q_{1}^n}|\nabla  Du_{k}(X_{1}+a,X_{2})|^{2}\xi_{1}^{\eta}(X_{1})\xi_{2}(X_{2})dx\quad \text{ for } a\in B_{1-\eta}^{n-4}.
		%$$
		
		% File: uniform_error_estimate_proof.tex
		% Include this file in your manuscript with:
		% \input{uniform_error_estimate_proof}
		%
		% The parent document should load at least:
		% \usepackage{amsmath,amssymb,amsthm,mathtools}
		
		%\subsection{Uniform estimates for the error terms}
		
		Write
		\[
		X=(X_1,X_2)\in \mathbb{R}^{n-4}\times\mathbb{R}^{4},
		\]
		and let
		\[
		\xi_1\in C_c^\infty(B_1^{n-4}),
		\qquad
		\xi_1\geq 0,
		\qquad
		\int_{\mathbb{R}^{n-4}}\xi_1(X_1)\,dX_1=1.
		\]
		For $\eta>0$, define
		\[
		\xi_1^\eta(X_1)
		:=
		\eta^{4-n}\xi_1\!\left(\frac{X_1}{\eta}\right).
		\]
		Then
		\[
		\operatorname{supp}\xi_1^\eta\subset B_\eta^{n-4}
		\qquad\text{and}\qquad
		\int_{\mathbb{R}^{n-4}}\xi_1^\eta(X_1)\,dX_1=1.
		\]
		Let $\xi_2\in C_c^\infty(B_1^4)$, and set
		\[
		F_k^\eta(a)
		:=
		\int_{Q_1^n}
		\left|\nabla Du_k(X_1+a,X_2)\right|^2
		\xi_1^\eta(X_1)\xi_2(X_2)\,dX,
		\qquad
		a\in B_{1-\eta}^{n-4}.
		\]
		
		%For $j=1,\ldots,n-4$, the stationary condition yields
		%\begin{equation}\label{eq:distributional equation for FGT}
		%	\partial_{a_j}F_k^\eta	=	\Div_a G_{k,j}^\eta+T_{k,j}^\eta
		%\end{equation}
		%in the sense of distributions, where the two terms arise from
		%\begin{align*}
		%	G_{k,j}^\eta(a)	%:=\left(-4\int_{Q_1^n}
		%	\Bigl(
		%	\nabla_jDu_k(X_1+a,X_2)
		%	\cdot
		%	\nabla_\beta Du_k(X_1+a,X_2)
		%	\Bigr)\times
		%	\bigl(\xi_1^\eta(X_1)\xi_2(X_2)\bigr)
		%	\,dX,\right)_{1\le\beta\le n-4}
		%\end{align*}
		%and
		%\begin{align*}
		%	T_{k,j}^\eta(a)
		%	:=
		%	-2\int_{Q_1^n}
		%	\Bigl(
		%	H_{\alpha\beta}(u_k)(X_1+a,X_2)
		%	\cdot
		%	\partial_j u_k(X_1+a,X_2)
		%	\Bigr)\times
		%	\partial_{\alpha\beta}
		%	\bigl(\xi_1^\eta(X_1)\xi_2(X_2)\bigr)
		%	\,dX.
		%\end{align*}
		%Here repeated Greek indices are summed from $1$ to $n$.
		
		%Lemma \ref{eq: eq 2.1 by VS} below shows that  $$\|G_{k}^{\eta}\|_{L^{1}(B_{1-\eta}^{n-4})}+\|T_{k}^{\eta}\|_{L^{1}(B_{1-\eta}^{n-4})}\to0$$ as $k\to\wq$ uniformly in $\eta$.

		For $j=1,\ldots,n-4$, the stationary condition yields
		$$
		\begin{aligned}
			\pa_{a_{j}}F_{k}^{\eta}(a) =&
			-4\sum_{\beta=1}^{n}\int_{Q_{1}^n}\left(\nabla_{j}Du_k\cdot \nabla_{\beta}Du_k\right) \, \partial_\beta\xi(X_1-a,X_2)\\
			&-2\sum_{\alpha,\beta=1}^{n}\int_{Q_{1}^n}\left(H_{\alpha\beta}(u_k) \cdot \partial_j u_k\right)\,\partial_{\alpha\beta} \xi(X_1-a,X_2)\\
			=& \sum_{\beta=1}^{n-4}\pa_{a_{\beta}}\Big(4\int_{Q_{1}^n}\left(\nabla_{j}Du_k\cdot \nabla_{\beta}Du_k\right)(X_1+a,X_2) \,\xi_{1}^{\eta}\xi_{2}\Big)\\
			& -4\sum_{\beta=n-3}^{n}\int_{Q_{1}}\left(\nabla_{j}Du_k\cdot \nabla_{\beta}Du_k\right)(X_{1}+a,X_{2})\,\xi_{1}^{\eta}(\pa_{\beta}\xi_{2})\\
			& -\sum_{\beta=1}^{n-4}\pa_{a_{\beta}}\left(2\sum_{\alpha=1}^n\int_{Q_{1}^n}\partial_\alpha\left(H_{\alpha\beta}(u_k) \cdot \partial_j u_k\right)(X_1+a,X_2)\,\xi_{1}^{\eta}\xi_{2}\right)\\	
			& +2\sum_{\beta=n-3}^{n}\sum_{\alpha=1}^n\int_{Q_{1}^n}\partial_\alpha\left(H_{\alpha\beta}(u_k) \cdot \partial_j u_k\right)(X_1+a,X_2)\,\xi_{1}^{\eta}(\pa_{\beta}\xi_{2})\\
			=:& \Div_aG_{k}^{\eta}(a)+T_{k}^{\eta}(a),
		\end{aligned}
		$$
		where
		$$
		\begin{aligned}
		G_{k}^{\eta}(a)=&\Bigg(4\int_{Q_{1}^n}\left(\nabla_{j}Du_k\cdot \nabla_{\beta}Du_k\right)(X_1+a,X_2) \,\xi_{1}^{\eta}\xi_{2}\\
		&-2\sum_{\alpha=1}^n\int_{Q_{1}^n}\partial_\alpha\left(H_{\alpha\beta}(u_k) \cdot \partial_j u_k\right)(X_1+a,X_2)\,\xi_{1}^{\eta}\xi_{2}\Bigg)_{1\le \beta\le n-4}
		\end{aligned}
		$$
		and
		$$
		\begin{aligned}
			T_{k}^{\eta}(a)=& -4\sum_{\beta=n-3}^{n}\int_{Q_{1}}\left(\nabla_{j}Du_k\cdot \nabla_{\beta}Du_k\right)(X_{1}+a,X_{2})\,\xi_{1}^{\eta}(\pa_{\beta}\xi_{2})\\
			& +2\sum_{\beta=n-3}^{n}\sum_{\alpha=1}^n\int_{Q_{1}^n}\partial_\alpha\left(H_{\alpha\beta}(u_k) \cdot \partial_j u_k\right)(X_1+a,X_2)\,\xi_{1}^{\eta}(\pa_{\beta}\xi_{2}).
		\end{aligned}
		$$
		Applying lemma \ref{lemma: interpolation Morrey Sobolev} to $\partial_j u_k$, together with \eqref{eq: Claim 1 eq} and \eqref{eq: X-2}, we conclude $$\|\partial_j u_k\|^4_{L^4(Q^n_1)}\le C\|\partial_j u_k\|^2_{M^{2,n-2}(Q^n_{3/2})}\left(\|\nabla \partial_j u_k\|^2_{L^2(Q^n_{3/2})}+\|\partial_j u_k\|^2_{L^2(Q^n_{3/2})} \right)\to 0.$$
		This implies $\|\partial_\alpha\left(H_{\alpha\beta}(u_k) \cdot \partial_j u_k\right)\|_{L^1}\to 0$ since $u\in W^{3,(4/3,1)}, H_{\alpha\beta}\in L^2$. We also have $\|\nabla_{j}Du_k\cdot \nabla_{\beta}Du_k\|_{L^1}\to0$ due to \eqref{eq: Claim 1 eq}. By Fubini's theorem, there holds $$\|G_{k}^{\eta}\|_{L^{1}(B_{1-\eta}^{n-4})}+\|T_{k}^{\eta}\|_{L^{1}(B_{1-\eta}^{n-4})}\to0$$
		as $k\to\wq$ uniformly in $\eta$.

		Thus, by Allard's strong
		constancy Lemma (see e.g. \cite[Lemma 2.7]{Lin-1999-Annals}) we deduce that
		\begin{equation}\label{eq: X-1}
			\|F_{k}^{\eta}-c_{k}^{\eta}\|_{L_{\loc}^{1}(B_{1-\eta}^{n-4})}\to0\qquad\text{as }k\to\wq 
		\end{equation}
		for some constant $c_{k}^{\eta}$ and the above convergence is uniform with respect to $\eta$. 
		
		On the other hand, it is straightforward to show that $F_{k}^{\eta}(a)\to F_{k}^{0}(a)$
		in $L_{\loc}^{1}(B^{n-4}_1)$ as $\eta\to0$, where
		$$
		F_{k}^{0}(a)=\int_{B_1^{4}}|\nabla Du_{k}(a,X_{2})|^{2}\xi_{2}(X_{2})dX_{2}, \quad  a\in B_1^{n-4}.
		$$
		Combining with \eqref{eq: X-1}, this implies that for each $k\gg 1$, $\{c_{k}^{\eta}\}_{\eta>0}$ is a Cauchy sequence as $\eta\to0$. Hence there exist $\{c_{k}\}\subset\R$ such
		that
		\begin{equation}\label{eq:consequence of Allard strong constancy lemma}
			F^0_k(a)\to c_k\quad\text{strongly in }L^1(B_{r}^{n-4}(X_{1}))\quad\text{as }k\to\infty
		\end{equation}
		for any $B_{r}^{n-4}(X_{1})\subset\subset B^{n-4}_1$.
		
		Finally, we claim that for any $B_{r}^{n-4}(X_{1})\subset\subset B^{n-4}_1$,  it holds
		\begin{equation}\label{eq:key for reduction}
			\lim_{k\to\wq}\int_{B_{r}^{n-4}(X_{1})}F_{k}^{0}(a)\,da =\Theta_{\nu}^{n-4}(0){\cal H}^{n-4}(B_{r}^{n-4}).
		\end{equation}
		
		%Recall that $|\De u_{i}|^{2}dx\wto\Theta_{\nu}^{n-4}(0){\cal H}^{n-4}\LLcorner B^{n-4}\times\{0\}$.
		%To continue, we further require that $\xi_{2}\equiv1$ on $B_{1/2}^{4}$.
		Indeed, for any $B_{r}^{n-4}(X_{1})\subset\subset B^{n-4}_1$, we select $\xi_1^k\in C_0^\infty(B_{r}^{n-4}(X_{1}))$ such that $\xi_1^j$ is monotonically increasing and converges to the constant function $1$, as $j\to\infty$, on $B_{r}^{n-4}(X_{1})$, and $\xi_{2}\equiv1$ on $B_{1/2}^{4}$.
		Let $$C_k^j=\int_{B_{r}^{n-4}(X_{1})}F_{k}^{0}(a)\xi_1^j(a)\,da.$$
		Then for each fixed $k$, $C_k^j$ is monotonically increasing with respect to $j$.
		Moreover, since $|\nabla Du_{k}|^{2}dx\wto\Theta_{\nu}^{n-4}(0){\cal H}^{n-4}\LLcorner  B_1^{n-4}\times\{0\}$, we have
		$$
		\begin{aligned}
			\lim_{k\to \infty}C_k^j&=\lim_{k\to \infty}\int_{B_{r}^{n-4}(X_{1})}F_{k}^{0}(a)\xi_1^j(a)da\\
			&=\lim_{k\to\wq}\int_{B_{r}^{n-4}(X_{1})\times B_1^{4}}|\nabla Du_{k}(x)|^{2}\xi_1^j(a)\xi_{2}(X_{2})dx\\
			&=\Theta_{\nu}^{n-4}(0)\int_{B_{r}^{n-4}(X_{1})}\xi_1^j(a)\,d{\cal H}^{n-4}(a)=:C^j.
		\end{aligned}
		$$
		By the monotone convergence theorem, we have
		$$
		\lim_{j\to \infty} C^j=\lim_{j\to \infty}\Theta_{\nu}^{n-4}(0)\int_{B_{r}^{n-4}(X_{1})\times\{0\}}\xi_1^j(a)d{\cal H}^{n-4}(a)=\Theta_{\nu}^{n-4}(0){\cal H}^{n-4}(B_{r}^{n-4}).
		$$
		On the other hand, the monotone convergence theorem gives
		$$
		\lim_{j\to\infty}C_k^j=\lim_{j\to \infty}\int_{B_{r}^{n-4}(X_{1})}F_{k}^{0}(a)\xi_1^j(a)da=\int_{B_{r}^{n-4}(X_{1})}F_{k}^{0}(a)da=:C_k,
		$$
		and
		$$\lim_{k\to \infty}C_k=(\lim_{k\to\infty}c_k){\cal H}^{n-4}(B_{r}^{n-4}).$$
		%Set $d^j=\int_{B_{r}^{n-4}(X_{1})}\xi_1^j(a)da$.
		Then by \eqref{eq:consequence of Allard strong constancy lemma} and our choice of $\xi_1^j$, we have
		$$
		\begin{aligned}
			\lim_{k\to\infty}\sup_{j}\left|C_k^j-c_k\int_{B_{r}^{n-4}(X_{1})}\xi_1^j(a)da\right|&\leq
			\lim_{k\to\infty}\sup_{j}\int_{B_{r}^{n-4}(X_{1})}\left|F_{k}^{0}(a)\xi_1^j(a)-c_k\xi_1^j(a)\right|\,da\\
			&\leq \lim_{k\to\infty}\int_{B_{r}^{n-4}(X_{1})}\left|F_{k}^{0}(a)-c_k\right|\,da=0.
		\end{aligned}
		$$
		This implies that as $k$ is large enough, $C_k^j$ is close to $c_k\int_{B_{r}^{n-4}(X_{1})}\xi_1^j(a)da$ uniformly in $j$. In particular, $C_k^j$ converges to $C^j$ as $k\to\infty$ uniformly in $j$. A simple fact in mathematical analysis (see \cite[Lemma 2.2]{Guo-Wang-Xiang-2026-EI for ESbiHM}) gives
		\[
		\lim_{k\to \infty}\lim_{j\to \infty}C_k^j=\lim_{j\to \infty}\lim_{k\to \infty}C_k^j.
		\]
		It follows
		$$
		\lim_{k\to\wq}\int_{B_{r}^{n-4}(X_{1})}F_{k}^{0}(a)da =\Theta_{\nu}^{n-4}(0){\cal H}^{n-4}(B_{r}^{n-4}).
		$$
		%\textbf{(?)}
		%\[
		%\begin{aligned}\lim_{i\to\wq}\int_{B_{r}^{n-4}(X_{1})}F_{i}^{0}(a)da & =\lim_{i\to\wq}\int_{B_{r}^{n-4}(X_{1})\times B^{4}}|\De u_{i}(x)|^{2}\xi_{2}(X_{2})dx\\
		% & =\Theta_{\nu}^{n-4}(0)\int_{B_{r}^{n-4}(X_{1})\times\{0\}}\xi_{2}(0)d{\cal H}^{n-4}\\
		% & =\Theta_{\nu}^{n-4}(0){\cal H}^{n-4}(B_{r}^{n-4}).
		%\end{aligned}
		%\]
		Hence, for any $B_{r}^{n-4}(X_{1})\subset\subset B^{n-4}$,
		$$
		\lim_{k\to\wq}\dashint_{B_{r}^{n-4}(X_{1})}F_{k}^{0}(a)da=\Theta_{\nu}^{n-4}(0).
		$$
		Note that
		$$
		\begin{aligned}\int_{B_{r}^{n-4}(X_{1})}F_{k}^{0}(a)\,da & =\int_{B_{r}^{n-4}(X_{1})\times\left(B_1^{4}\backslash B_{1/2}^{4}\right)}|\nabla Du_{k}(a,X_{2})|^{2}\xi_{2}(X_{2})\,dadX_{2}\\
			& \quad+\int_{B_{r}^{n-4}(X_{1})\times B_{1/2}^{4}}|\nabla Du_{k}(a,X_{2})|^{2}\,dadX_{2}.
		\end{aligned}
		$$
		In the right hand side, the first term vanishes as $k\to\wq$, since $\nabla Du_{k}\to0$ uniformly
		on $B_{r}^{n-4}(X_{1})\times\left(B^{4}_1\backslash B_{1/2}^{4}\right)$.
		Hence
		$$
		\lim_{k\to\wq}\dashint_{B_{r}^{n-4}(X_{1})}\int_{B_{1/2}^{4}}|\nabla Du_{k}(X_{1},X_{2})|^{2}dX_{2}dX_{1}=\Theta_{\nu}^{n-4}(0).
		$$holds for all $B_{r}^{n-4}(X_{1})\subset B_1^{n-4}$. The proof is complete.
	\end{proof}

	In the proof above, we have used the following interpolation inequality.
	Let $Q_r^n=(-r,r)^n$. For an open set $\Omega\subset\R^n$, define the local Morrey norm by
	\begin{equation}
		\|f\|_{M^{2,n-2}(\Omega)}^2
		:=\sup_{x\in\Omega}\sup_{0<r\leq 1}
		r^{2-n}\int_{B_r(x)\cap\Omega}|f(y)|^2dy.
		\label{eq:morrey2}
	\end{equation}
	
	\begin{lemma}[Local Morrey--Sobolev interpolation]\label{lemma: interpolation Morrey Sobolev}
		Let $n\geq2$. There exists a constant $C=C(n)>0$ such that every
		\[
		u\in W^{1,2}(B_1^n)\cap M^{2,n-2}(B_1^n)
		\]
		satisfies
		\begin{equation}
			\|u\|_{L^4(B_{1/2}^n)}^4
			\leq C
			\|u\|_{M^{2,n-2}(B_1^n)}^2
			\|u\|_{W^{1,2}(B_1^n)}^2.
			\label{eq:ball-estimate}
		\end{equation}
		More generally, there is a constant $C=C(n)>0$ such that every
		\[
		u\in W^{1,2}(Q_{3/2}^n)\cap M^{2,n-2}(Q_{3/2}^n)
		\]
		satisfies
		\begin{equation}\label{eq:cube-estimate}
			\|u\|_{L^4(Q_1^n)}^4
			\leq{}C\|u\|_{M^{2,n-2}(Q_{3/2}^n)}^2 
			\times\left(
			\|\nabla u\|_{L^2(Q_{3/2}^n)}^2
			+\|u\|_{L^2(Q_{3/2}^n)}^2
			\right).
		\end{equation}
	\end{lemma}
	\begin{proof}
		Applying \cite[Equation (2.1)]{VanSchaftingen-14} with $k=1$, $\ell=0$, $p=2$, $q=4$, $\rho=1$ and $\lambda=1$, we obtain that for $v\in W^{1,2}(\R^n)$, it holds
		\begin{equation}
			\|v\|_{L^4(\R^n)}^4
			\leq{}C
			\left(
			\sup_{x\in\R^n}\sup_{0<r\leq1}
			r^{1-n}\int_{B_r(x)}|v(y)|d y
			\right)^2
			\quad\times\left(
			\|\nabla v\|_{L^2(\R^n)}^2
			+\|v\|_{L^2(\R^n)}^2
			\right).
			\label{eq:VS-special}
		\end{equation}
		On the other hand, it follows from the Cauchy-Swartz inequality that
		\begin{equation}
			\sup_{x\in\R^n}\sup_{0<r\leq1}
			r^{1-n}\int_{B_r(x)}|v(y)|dy
			\leq C(n)\|v\|_{M^{2,n-2}(\R^n)}.
			\label{eq:M2-to-M1}
		\end{equation}
		Combining \eqref{eq:VS-special} and \eqref{eq:M2-to-M1} gives the global estimate
		\begin{equation}\label{eq:whole-M2}
			\|v\|_{L^4(\R^n)}^4
			\leq C(n)\|v\|_{M^{2,n-2}(\R^n)}^2
			\left(
			\|\nabla v\|_{L^2(\R^n)}^2
			+\|v\|_{L^2(\R^n)}^2
			\right).
		\end{equation}
		
		Next, we show the local estimate via standard cut-off argument. 
		Choose $\chi\in C_c^\infty(B_1^n)$ such that
		\[
		0\leq\chi\leq1,
		\qquad \chi\equiv1\text{ on }B_{1/2}^n,
		\qquad |\nabla\chi|\leq C(n).
		\]
		Define $v=\chi u$ on $B_1$ and extend $v$ by zero to $\R^n$. Since $\chi$ has compact support in $B_1$, the zero extension belongs to $W^{1,2}(\R^n)$.
		
		Because $v=u$ on $B_{1/2}$, we have
		\begin{equation}
			\|u\|_{L^4(B_{1/2})}^4
			\leq\|v\|_{L^4(\R^n)}^4.
			\label{eq:L4-localize}
		\end{equation}
		Note that
		\[
		\nabla v=\chi\nabla u+u\nabla\chi.
		\]
		Using the elementary inequality $|a+b|^2\leq2|a|^2+2|b|^2$, we obtain
		\begin{equation}\label{eq:grad-cutoff}
			\|\nabla v\|_{L^2(\R^n)}^2
			\leq2\|\chi\nabla u\|_{L^2(B_1)}^2
			+2\|u\nabla\chi\|_{L^2(B_1)}^2 
			\leq C(n)\left(
			\|\nabla u\|_{L^2(B_1)}^2+
			\|u\|_{L^2(B_1)}^2
			\right),
		\end{equation}
		and
		\begin{equation}
			\|v\|_{L^2(\R^n)}^2\leq\|u\|_{L^2(B_1)}^2.
			\label{eq:L2-cutoff}
		\end{equation}
		
		It remains to control the global Morrey norm of $v$. Let $x\in\R^n$ and $0<r\leq1$. Since $|v|\leq|u|$ on $B_1$ and $v=0$ outside $B_1$,
		\[
		\int_{B_r(x)}|v|^2
		=\int_{B_r(x)\cap B_1}|\chi u|^2
		\leq\int_{B_r(x)\cap B_1}|u|^2.
		\]
		If $B_r(x)\cap\supp\chi=\varnothing$, the left side is zero. Otherwise choose $z\in B_r(x)\cap\supp\chi\subset B_1$. Then
		\[
		B_r(x)\cap B_1\subset B_{2r}(z)\cap B_1.
		\]
		For $2r\leq1$, the definition \eqref{eq:morrey2} gives
		\begin{align*}
			r^{2-n}\int_{B_r(x)}|v|^2
			&\leq r^{2-n}\int_{B_{2r}(z)\cap B_1}|u|^2\\
			&\leq r^{2-n}(2r)^{n-2}
			\|u\|_{M^{2,n-2}(B_1)}^2\\
			&=2^{n-2}\|u\|_{M^{2,n-2}(B_1)}^2.
		\end{align*}
		If $1/2<r\leq1$, then $r^{2-n}\leq C(n)$ and
		\[
		r^{2-n}\int_{B_r(x)}|v|^2
		\leq C(n)\|u\|_{L^2(B_1)}^2.
		\]
		The radius-one term in the local Morrey norm controls $\|u\|_{L^2(B_1)}^2$ up to a dimensional constant. Hence
		\begin{equation}
			\|v\|_{M^{2,n-2}(\R^n)}
			\leq C(n)\|u\|_{M^{2,n-2}(B_1)}.
			\label{eq:morrey-cutoff}
		\end{equation}
		
		Applying \eqref{eq:whole-M2} to $v$, and then \eqref{eq:L4-localize}, \eqref{eq:grad-cutoff}, \eqref{eq:L2-cutoff}, and \eqref{eq:morrey-cutoff}, we thus obtain
		\[
		\begin{aligned}
			\|u\|_{L^4(B_{1/2})}^4
			&\leq\|v\|_{L^4(\R^n)}^4
			\leq C\|v\|_{M^{2,n-2}(\R^n)}^2
			\left(\|\nabla v\|_2^2+\|v\|_2^2\right)\\
			&\leq C\|u\|_{M^{2,n-2}(B_1)}^2
			\left(\|\nabla u\|_{L^2(B_1)}^2+
			\|u\|_{L^2(B_1)}^2\right)\\
			&=C\|u\|_{M^{2,n-2}(B_1)}^2
			\|u\|_{W^{1,2}(B_1)}^2.
		\end{aligned}
		\]
		This proves \eqref{eq:ball-estimate}.
		
		For the cube estimate, choose $\chi\in C_c^\infty(Q_{3/2}^n)$ such that
		\[
		0\leq\chi\leq1,
		\qquad \chi\equiv1\text{ on }Q_1^n,
		\qquad |\nabla\chi|\leq C(n).
		\]
		Set $v=\chi u$ and extend $v$ by zero to $\R^n$. Exactly as above,
		\[
		\|u\|_{L^4(Q_1)}^4\leq\|v\|_{L^4(\R^n)}^4,
		\]
		\[
		\|\nabla v\|_{L^2(\R^n)}^2+
		\|v\|_{L^2(\R^n)}^2
		\leq C(n)\left(
		\|\nabla u\|_{L^2(Q_{3/2})}^2+
		\|u\|_{L^2(Q_{3/2})}^2
		\right),
		\]
		and the fixed cutoff and the positive distance between $Q_1$ and $\partial Q_{3/2}$ imply
		\[
		\|v\|_{M^{2,n-2}(\R^n)}
		\leq C(n)\|u\|_{M^{2,n-2}(Q_{3/2})}.
		\]
		Using \eqref{eq:whole-M2}, we infer that
		\[
		\begin{aligned}
			\|u\|_{L^4(Q_1)}^4
			&\leq C\|v\|_{M^{2,n-2}(\R^n)}^2
			\left(\|\nabla v\|_{L^2(\R^n)}^2+
			\|v\|_{L^2(\R^n)}^2\right)\\
			&\leq C\|u\|_{M^{2,n-2}(Q_{3/2})}^2
			\left(
			\|\nabla u\|_{L^2(Q_{3/2})}^2+
			\|u\|_{L^2(Q_{3/2})}^2
			\right).
		\end{aligned}
		\]
		This proves \eqref{eq:cube-estimate}.
	\end{proof}

	\subsection{$L^{(2,\wq)}$-estimate of $\sum_{\alpha,\beta=n-3}^{n}H_{\alpha\beta}(u_{k})(X_{1}^{k},X_{2})$}
	With the reduction Lemma \ref{lem: main lemma} at hand, we find that
	$$
	\begin{aligned}
		&\sum_{\alpha,\beta=n-3}^{n}\int_{B_{1/2}^{4}(X_{2}^{k})}\left|H_{\alpha\beta}(u_k)(X_{1}^{k},X_{2})\right|^{2}\,dX_{2}\\
		=&\sum_{\alpha,\beta=n-3}^{n}\Big(\int_{B_{1/2}^{4}(X_{2}^{k})\backslash B_{\de_{k}R}^{4}(X_{2}^{k})}+\int_{B_{\de_{k}R}^{4}(X_{2}^{k})}\Big)\Big|H_{\alpha\beta}(u_k)(X_{1}^{k},X_{2})\Big|^{2}\,dX_{2}\\
		=&\sum_{\alpha,\beta=n-3}^{n}\int_{A(R,k)}\left|H_{\alpha\beta}(u_k)(X_{1}^{k},X_{2})\right|^{2}dX_{2}
		+\sum_{\alpha,\beta=n-3}^{n}\int_{B_{R}^{4}}|H_{\alpha\beta}(v_{k})(0,Y_{2})|^{2}\,dY_2,
	\end{aligned}
	$$
	where $v_{k}(y)=u_k((X_1^{k}, X_2^k)+\de_{k}Y)$ and
	\[
	A(R,k)=B_{1/2}^{4}(X_{2}^{k})\backslash B_{\de_{k}R}^{4}(X_{2}^{k}).
	\]
	By Lin-Riviere's standard reduction procedure (see \cite{Lin-Riviere-2002}), we may assume there is only one bubble at $(0,0)$. In this case, since $v_{k}$ converges to a bubble $v\in H_{\loc}^{2}(\R^{4})\cap C_{\loc}^{2}(\R^{4})$, the energy identity would follow from
	\begin{equation}\label{eq: X-4}
		\lim_{R\to\wq}\lim_{k\to\wq}\sum_{\alpha,\beta=n-3}^{n}\int_{A(R,k)}|H_{\alpha\beta}(u_k)(X_{1}^{k},X_{2})|^{2}\,dX_{2}=0.
	\end{equation}
	
	We will prove \eqref{eq: X-4} by applying the duality between the Lorentz spaces $L^{(2,1)}$ and $L^{(2,\wq)}$. $L^{(2,1)}$-estimate has been done in \eqref{eq: Dj uk L(4/j,1)} and now we need to estimate the $L^{(2,\wq)}$-norm of $H_{\alpha\beta}(u_k)(X_{1}^{k},X_{2})$ on $A(R,k)$.
	
	\begin{lemma}\label{lem: small endpoint}
		For any $R>1$ and $0<\la\ll1/2$ such that $\de_{k}R/\la<1/2$, there holds,
		$$
		\lim_{k\to\wq}\sum_{\alpha,\beta=n-3}^{n}\int_{B_{1/2}^{4}(X_{2}^{k})\backslash B_{\la}^{4}(X_{2}^{k})}|H_{\alpha\beta}(u_k)(X_{1}^{k},X_{2})|^{2}\,dX_{2}=0,
		$$
		and
		$$
		\lim_{R\to\wq}\lim_{k\to\wq}\sum_{\alpha,\beta=n-3}^{n}\int_{B_{\de_{k}R/\la}^{4}(X_{2}^{k})\backslash B_{\de_{k}R}^{4}(X_{2}^{k})}|H_{\alpha\beta}(u_k)(X_{1}^{k},X_{2})|^{2}\,dX_{2}=0.
		$$
	\end{lemma}

	\begin{proof} The first equation holds  since $|\nabla Du_{k}|^{2}\to0$ uniformly on $B_{1}^{n}(0)\backslash B_{\la/2}^{4}(0)$ and $X_{2}^{k}\to0$ as $k\to\infty$. 
		
		On the other hand, recall that the map  $v_{k}(Y_{1},Y_{2})=u_{k}(X_{1}^{k}+\de_{k}Y_{1},X_{2}^{k}+\de_{k}Y_{2})$
		converges to $v\in H^{2}(\R^{4},N)$ in $C_{\loc}^{2}(\R^n)$.
		Thus
		$$
		\begin{aligned} 
			&\sum_{\alpha,\beta=n-3}^{n}\int_{B_{\de_{k}R/\la}^{4}(X_{2}^{k})\backslash B_{\de_{k}R}^{4}(X_{2}^{k})}|H_{\alpha\beta}(u_k)(X_{1}^{k},X_{2})|^{2}\,dX_{2}\\
			=&\sum_{\alpha,\beta=n-3}^{n}\int_{\{R\le|Y_{2}|\le R/\la\}}|H_{\alpha\beta}(v_k)(0,Y_{2})|^{2}\,dY_{2}\\
			\to&\int_{\{R\le|Y_{2}|\le R/\la\}}|\nabla Dv|^{2}\,dY_{2},
		\end{aligned}
		$$
		as $k\to\wq$. The second assertion follows from $v\in H^2(\R^{4})$.
	\end{proof}
	Now we prove the following key estimate.
	
	\begin{proposition} \label{prop: key-estimate-1}
		For any $\e>0$, there exists $R_{0}$ sufficiently large such that
		\begin{equation}\label{no-neck-bubble}
			\limsup_{k\to\wq}\sup_{\rho\in(\de_{k}R,1/4)}I(\rho,u_{k},X^{k})\le\e
		\end{equation}
		for $R\ge R_{0}$, where $X^k=(X_1^k,X_2^k)$ and
		$$
		I(\rho,u_{k},X^{k})=\rho^{4-n}\int_{B_{\rho}^{n-4}(X_{1}^{k})}\int_{B_{2\rho}^{4}(X_{2}^{k})\backslash B_{\rho}^{4}(X_{2}^{k})}\left(|\nabla Du_{k}|^{2}+\rho^{-2}|Du_{k}|^{2}\right)\,dX.
		$$
	\end{proposition}
	
	Since $\sup_{\rho\in(\de_{k}R,1/4)}I(\rho,u_{k},X^{k})$ is monotonically decreasing with respect to $R$, according to \eqref{no-neck-bubble} we have 
	$$
	\lim_{R\to\wq}\limsup_{k\to\wq}\sup_{\rho\in(\de_{k}R,1/4)}I(\rho,u_{k},X^{k})=0.
	$$
	
	\begin{proof}
		We prove by contradiction. Suppose there exists $\e_{1}>0$, such that for any sufficiently large $R$, there exist $\rho_{k}\in(\de_{k}R,1/4)$ so that
		\begin{equation}\label{eq: X-6}
			I(\rho_{k},u_{k},X^{k})=\sup_{\rho\in(\de_{k}R,1/4)}I(\rho,u_{k},X^{k})\ge\e_{1}.
		\end{equation}
		
		\noindent\textbf{Claim.} There hold both $\rho_{k}\to0$ and $\de_{k}R/\rho_{k}\to0$ as $k\to\wq$. 
		
		Indeed, if $\rho_{k}\ge c>0$ as $k\to\wq$, then $$B_{2\rho_{k}}^{4}(X_{2}^{k})\backslash B_{\rho_{k}}^{4}(X_{2}^{k})\subset B^{4}(0)\backslash B_{c/2}^{4}(0),$$
		which  implies $(|\nabla Du_{k}|^{2}+\rho_{k}^{-2}|Du_{k}|^{2})$
		converges uniformly to $0$ on $B_{2\rho_{k}}^{4}(X_{2}^{k})\backslash B_{\rho_{k}}^{4}(X_{2}^{k})$.
		This implies that $I(\rho_{k},u_{k},X^{k})\to0$ as $k\to\wq$, which contradicts to  (\ref{eq: X-6}).
		
		Moreover, since $\rho_{k}\ge\de_{k}R$, if $\rho_{k}/(\de_{k}R)\to c\in[1,\wq)$,
		then by considering $v_{k}(Y)=u_{k}(X_1^{k}+\de_{k}Y_1,X_2^{k}+\de_{k}Y_2)$ we deduce
		$$
		\begin{aligned}
			\e_{1}  \le&\left(\frac{\rho_{k}}{\de_{k}}\right)^{4-n}\int_{B_{\rho_{k}/\de_{k}}^{n-4}}\int_{B_{2\rho_{k}/\de_{k}}^{4}\backslash B_{\rho_{k}/\de_{k}}^{4}}
			\Big(|\nabla Dv_{k}|^{2}+\left(\frac{\rho_{k}}{\de_{k}}\right)^{-2}|Dv_{k}|^{2}\Big)\,dY\\
			\to&(cR)^{4-n}\int_{B_{cR}^{n-4}}\int_{B_{2cR}^{4}\backslash B_{cR}^{4}}\Big(|\nabla Dv|^{2}+\left(cR\right)^{-2}|Dv|^{2}\Big)\,dY\\
			=&\om_{n-4}\int_{B_{2cR}^{4}\backslash B_{cR}^{4}}\Big(|\nabla Dv|^{2}+\left(cR\right)^{-2}|Dv|^{2}\Big)\,dY_{2}.
		\end{aligned}
		$$
		Since $v\in H^{2}(\R^{4})$, this converges to zero as $R\to\wq$, which contradicts \eqref{eq: X-6} again. The claim is proved.
		~\\
		
		Now, define $w_{k}(Y)=u_{k}(X^{k}+\rho_{k}Y)$ for $|Y|\le1/(2\rho_{k})\to\wq$.
		Then the estimate (\ref{eq: uniform morrey norm}) implies that,
		for any $r>0$ and all $k\gg1$,
		\begin{equation}\label{eq: X-7}
			\begin{aligned} & r^{4-n}\int_{B_{r}^{n}}\Big(|\nabla Dw_{k}|^{2}+r^{-2}|Dw_{k}|^{2}\Big)\,dY\\
				=&(\rho_{k}r)^{4-n}\int_{B^{n}_{\rho_{k}r}(X^{k})}\Big(|\nabla Du_{k}|^{2}+(\rho_{k}r)^{-2}|Du_{k}|^{2}\Big)\,dY\le C,
			\end{aligned}
		\end{equation}
		which means $w_{k}$ is locally uniformly bounded in $H^{2}$. Moreover, \eqref{eq: X-6} is equivalent to
		\begin{equation}\label{eq: normalization}
			\int_{B_1^{n-4}}\int_{B_{2}^{4}\backslash B_1^{4}}\left(|\nabla Dw_{k}|^{2}+|Dw_{k}|^{2}\right) dY \ge \e_{1}
		\end{equation}
		for all $k\gg1$. Hence, up to a subsequence if necessary, we assume that
		$$
		w_{k}\wto w_{\wq} \ {\rm{in}}\ W_{\loc}^{2,2}(\R^{n})\ \ {\rm{ and \ }} \ \ |\nabla Dw_{k}|^{2}dY\wto|\nabla Dw_{\wq}|^{2}dY+\nu_{2}$$
		as weak convergence of Radon measures for some nonnegative defect measure
		$\nu_{2}$.
		
		If $\nu_{2}=0$, then we have $w_{i}\to w_{\wq}$ strongly in $W_{\loc}^{2,2}(\R^{n})$. By  \eqref{eq: X-2}, for any $r>0$ we have
		$$
		r^{4-n}\int_{B_{r}^{n-4}\times B_{1/\rho_{k}}^{4}}\sum_{j=1}^{n-4}|\na\pa_{j}w_{k}|^{2}=(\rho_{k}r)^{4-n}\int_{B_{\rho_{k}r}^{n-4}\times B_{1}^{4}}\sum_{j=1}^{n-4}|\na\pa_{j}u_{k}|^{2}\to0
		$$
		as $k\to\wq$. By \eqref{eq: normalization}, $w_{\wq}$ is a nontrivial map. Hence we conclude from these two facts and the estimate \eqref{eq: X-7} that $w_{\wq}(Y)=w_{\wq}(Y_{2}):\R^{4}\to N$ is a nontrivial, smooth intrinsic biharmonic mapping from $\R^{4}$ into $N$ with a bounded energy.
		
		If $\nu_{2}$ is nontrivial, then we may repeat the same blow up procedure to find another bubble.
		
		Both cases contradict to the assumption that there is only one bubble. The proof is complete.
	\end{proof}
	
	Now we are ready to prove Theorem \ref{thm: main results}.
	\begin{proof}[Proof of Theorem \ref{thm: main results}]
		By Proposition \ref{prop: key-estimate-1} and the $\varepsilon$-energy regularity theorem, for $\e>0$ sufficiently small, we infer that
		$$
		|X_{2}-X_{2}^{k}|^{2}|D^{2}u_{k}(X_{1}^{k},X_{2})|\le C\e
		$$
		on  $\{X_{1}^{k}\}\times \big(B_{1/4}^{4}\backslash B_{\delta_kR}^{4}\big)$. This implies that the set
		$$A_{\la}:=\{x\in A(R,k):|\na^{2}u_{k}(X_{1}^{k},X_{2})|>\la\}$$
		satisfies
		$$
		|X_{2}-X_{2}^{k}|\le\frac{C\sqrt{\e}}{\sqrt{\la}}\qquad\forall\,X_{2}\in A_{\la}.
		$$
		This implies $A_{\la}\subset B^{4}_{C\e/\sqrt{\la}}(X_{2}^{k})$. Thus
		$$
		\|D^{2}u_{k}\|_{L^{(2,\wq)}(\{X_{1}^{k}\}\times A(R,k))}=\sup_{\la>0}\la|A_{\la}|^{1/2}\le C\e^2.
		$$
		By the duality between  $L^{(2,1)}$ and $L^{(2,\infty)}$, we have
		$$
		\begin{aligned}
			&\quad \sum_{\alpha,\beta=n-3}^n \int_{A(R,k)}|H_{\alpha\beta}(u_{k})(X_{1}^{k},X_{2})|^{2}\,dX_{2}\\
			&\le C\|D^{2}u_{k}\|_{L^{(2,1)}(\{X_{1}^{k}\}\times A(R,k))}^2\|D^{2}u_{k}\|_{L^{(2,\wq)}(\{X_{1}^{k}\}\times A(R,k))}^2
			\le C\e^2.
		\end{aligned}
		$$
		Since $\e>0$ is arbitrary, \eqref{eq: X-4} follows and hence the proof is complete.
	\end{proof}

\end{document}